\documentclass[11pt]{article}
\usepackage{geometry}                % See geometry.pdf to learn the layout options. There are lots.
\usepackage{natbib}

\usepackage{graphicx}
\usepackage{amssymb}
\usepackage{amsthm}
\usepackage{mathrsfs}
\usepackage{wrapfig}
\usepackage{sidecap}
\usepackage{hyperref}

\usepackage{epstopdf}
\usepackage{fullpage}
\usepackage[latin1]{inputenc}
\usepackage{amsmath}
\usepackage{amsfonts}
\usepackage{amssymb}
\usepackage{graphicx}
\usepackage{comment}

\newcommand{\R}{{\mathbb{R}}}
\newcommand{\E}{{\mathbb{E}}}
\newcommand{\pr}{{\mathbb{P}}}
\newcommand{\G}{{\mathbb{G}}}

\newcommand{\T}{{\mathbb{T}}}
\newcommand{\C}{{\mathbb{C}}}
\newcommand{\tp}{$(\T,P)\textit{-walk on } G$ }
\newcommand{\tpns}{$(\T,P)\textit{-walk on } G$}

\newcommand{\atb}{\textsc{a-tree-bootstrap} }
\newcommand{\achb}{\textsc{a-chain-bootstrap} }
\newcommand{\atbns}{\textsc{a-tree-bootstrap}}
\newcommand{\achbns}{\textsc{a-chain-bootstrap}}
\newcommand{\utb}{\textsc{u-tree-bootstrap} }
\newcommand{\utbns}{\textsc{u-tree-bootstrap}}

\newcommand{\ssb}{\textsc{ss-bootstrap} }
\newcommand{\ssbns}{\textsc{ss-bootstrap}}

\newcommand{\gn}{N}
\newcommand{\tn}{n}
\newcommand{\D}{{\mathscr{D}}}
\newcommand{\F}{{\mathscr{F}}}
\newcommand{\A}{{\mathscr{A}}}

\newcommand{\cb}{\underline{c}}
\newcommand{\ct}{\bar{c}}

\newtheorem{lemma}{Lemma}[section]
\newtheorem{theorem}{Theorem}[section]
\newtheorem{definition}{Definition}

\newtheorem{corollary}{Corollary}[section]
\newtheorem{prop}{Proposition}[section]
\newtheorem{remark}{Remark}[section]
\newtheorem{fact}{Fact}[section]

\title{Network driven sampling; a critical threshold for design effects}
\author{Karl Rohe\footnote{Thank you Zoe Russek and Emma Krauska for recording the referral trees used in Figure \ref{fig:bothTrees} and helpful comments on this draft.  Thank you Bret Hanlon, Mohammad Khabbazian, Matthew Salganik, Mark Handcock, Sebastien Roch, Quansheng Liu, Ting Fung Ma, Arash Amini, Erik Volz, and Russell Lyons for thoughtful and helpful discussions over the course of this research. This research is supported by NSF grant DMS-1309998, DMS-1612456, and ARO grant W911NF-15- 1-0423.}, \\ University of Wisconsin - Madison}
\begin{document}
\maketitle

\begin{abstract}

Web crawling, snowball sampling, and respondent-driven sampling (RDS) are three types of network sampling techniques used to contact individuals in hard-to-reach populations. This paper studies these procedures as a Markov process on the social network that is indexed by a tree. Each node in this tree corresponds to an observation and each edge in the tree corresponds to a referral. Indexing with a tree (instead of a chain) allows for the sampled units to refer multiple future units into the sample.  

In survey sampling, the design effect characterizes the additional variance induced by a novel sampling strategy. If the design effect is some value $DE$, then constructing an estimator from the novel design makes the variance of the estimator $DE$ times greater than it would be under a simple random sample with the same sample size $n$. Under certain assumptions on the referral tree, the design effect of network sampling has a critical threshold that is a function of the referral rate $m$ and the clustering structure in the social network, represented by the second eigenvalue of the Markov transition matrix, $\lambda_2$. If $m < 1/\lambda_2^2$, then the design effect is finite (i.e. the standard estimator is $\sqrt{n}$-consistent). However, if $m > 1/\lambda_2^2$, then the design effect grows with $n$ (i.e. the standard estimator is no longer $\sqrt{n}$-consistent). Past this critical threshold, the standard error of the estimator converges at the slower rate of $n^{\log_m \lambda_2}$.  The Markov model allows for nodes to be resampled; computational results show that the findings hold in without-replacement sampling.
To estimate confidence intervals that adapt to the correct level of uncertainty, a novel resampling procedure is proposed.  Computational experiments compare this procedure to previous techniques.

\end{abstract}
%\begin{keyword}[class=MSC]
%\kwd[Primary ]{62D99}
%\kwd[; secondary ]{60J20}
%\end{keyword}
%
%\begin{keyword}
%\kwd{Stochastic Blockmodel}
%\kwd{social network}
%\kwd{link-tracing}
%\kwd{Galton-Watson}
%\end{keyword}

%\end{frontmatter}

\section*{Introduction}

This paper is motivated by respondent-driven sampling (RDS), a popular technique to sample marginalized and/or hard-to-reach populations \citep{heckathorn1997respondent}. 
RDS has become particularly popular in HIV research because the populations most at risk for HIV (i.e. people who inject drugs, female sex workers, and men who have sex with men) cannot be sampled using conventional techniques.  Several  domestic and international institutions use RDS to quantify the prevalence of HIV in at risk populations, including the Centers for Disease Control and Prevention  (CDC), the World Health Organization (WHO), and the Joint United Nations Programme on HIV/AIDS (UNAIDS) \citep{WHO}.  It has been applied in over 460 different studies, in 69 different countries \citep{white2015strengthening}. 

The RDS  process starts with a convenience sample of ``seeds" from the target population.  
They form wave zero.  
These participants are incentivized to (1) participate in the study and (2) pass three (or sometimes up to five) referral coupons to their friends.  The friends that return to the study site with a coupon form the first wave of the RDS.
The process iterates until the procedure reaches the target sample size, or until the process dies because participants stop passing coupons.  Figure \ref{fig:graphSeq} gives an illustration of this process.

\begin{figure}[t] %  figure placement: here, top, bottom, or page
   \centering
   \includegraphics[width=6in]{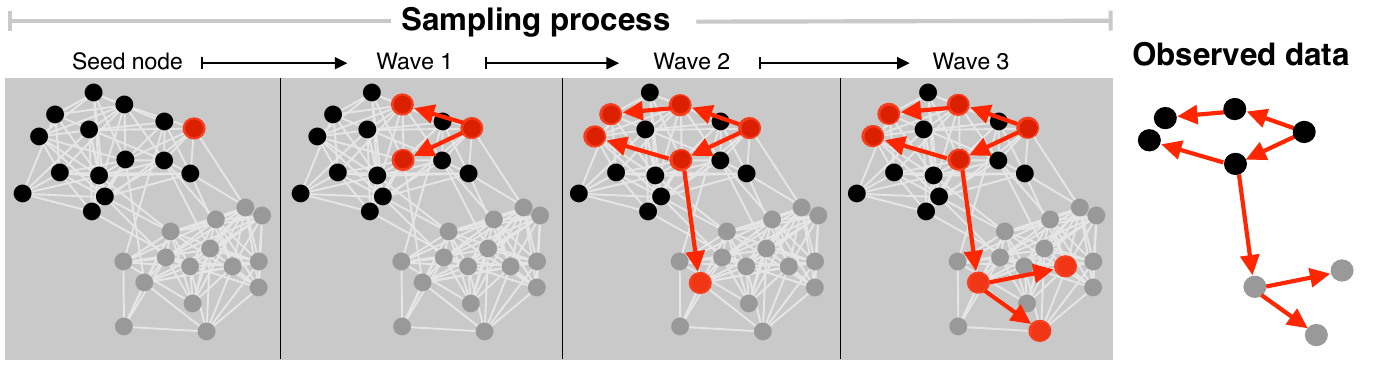} 
   \caption{Network sampling has two graphs: the underlying social network and the referral tree. Each node in the social network has some feature (e.g. HIV status).  In this diagram, the node feature is denoted by color.  When we sample a node, we observe (i) the node's color and (ii) which node referred the node into the sample.  In the end, we want to estimate the proportion of nodes that are grey.}
   \label{fig:graphSeq}
\end{figure}

If we presume that each participant refers a random subset of their friends,\footnote{In current implementations of RDS, randomization is not produced by researchers. Rather, it is presumed that people refer friends randomly.  The validity of such assumptions has been studied in several several ways in empirical and statistical papers.   For example, \cite{gile2015diagnostics} proposed statistical diagnostics to examine the convergence properties; \cite{arayasirikul2015qualitative} performed qualitative follow-up interviews to ask participants about difficulties in finding referrals; and \cite{mccreesh2012evaluation} compared a respondent-driven sample in Uganda with a total population survey on the same population.
} 
then RDS is a stochastic process on the members of the social network.
In the RDS literature, it is common to assume a Markov model because it is analytically tractable.  
The Markovian assumption is knowingly incorrect in practice.  
For example, it samples with-replacement; in practice, the sampling is performed without-replacement.  
Simulation studies suggest that, when the sample size is much smaller than the population size, the Markov model provides an approximation to more accurate simulation models  \citep{lu2012sensitivity}.
Under the Markov model, \cite{salganik2004sampling} and  \cite{volz2008probability} construct unbiased estimators.  While they are unbiased, they often suffer from high variability \citep{goel2009respondent, goel2010assessing}.  
In particular, \cite{goel2010assessing} shows in a wide range of computer experiments that (1) RDS
%\footnote
%{
%The acronym RDS is used to denote both (i) sampling and (ii) a sample.
%}  
often produces estimators with exceedingly large variance and (2) the popular bootstrap technique in \cite{salg} produced nominal $95\%$ confidence intervals with coverage probabilities between  $40\%$ and $70\%$.  
%These inadequacies have been further observed in \cite{neely2009statistical},  \cite{verdery2013network}, and \cite{baraff2016estimating}.  
This paper aims to build on these earlier results to provide a rigorous description of the inadequacies.

This paper focuses on one particular assumption of the Markov chain model which has received insufficient scrutiny.  In practice, each participant can refer between zero and three (sometimes up to five) future participants.  However, in the Markov chain model, each participant refers exactly one individual.  In the previous simulation study of \cite{goel2010assessing}, the ``chain" assumption was relaxed, while the ``Markov" assumption was retained.  This model has drastically different behaviors.  The results below show that this ``Markov tree" model remains analytically tractable.

The paper is organized as follows.  Section \ref{sec:prelim} defines the Markov model, the quantity we wish to estimate, and the estimators.  Section \ref{sec:var} provides an exact formula for the variance of an RDS estimator in Theorem \ref{thm:var}.
Section \ref{sec:critical} specifies the asymptotic behavior of the design effect in Theorem \ref{thm:critical}.
Section \ref{sec:replacement} studies the rate at which the Markov model resamples nodes in Theorem \ref{thm:upperRn}.  
With Theorems \ref{thm:var} and \ref{thm:critical}, Section \ref{sec:reinterpreting} reinterprets the previous simulation results in \cite{goel2010assessing}. Section \ref{sec:bootstrap} proposes a novel resampling technique \atb and compares it to previous techniques in computational experiments.  Finally, Section \ref{sec:discuss} concludes the paper.  All proofs are contained in the appendix.

\section{Preliminaries} \label{sec:prelim}

The model described below is a straightforward combination of the Markov models developed in the previous literature (e.g. \cite{heckathorn1997respondent, salganik2004sampling, volz2008probability} and \cite{goel2009respondent}). 
There are four necessary mathematical pieces:
a social network represented as a graph, a Markov transition matrix on the nodes of the graph, a referral tree to index the Markov process on the graph, and finally, a node feature defined for each node in the graph.

\subsection{Markov processes on a graph} \label{sec:graph}
%A node set  the edge set  contains the friendships. Together, they form 
A social network $G = (V,E)$ consists of the set of people $V = \{1, \dots, N\}$ and the set of friendships $E = \{(i,j) : \mbox{ $i$ and $j$ are friends}\}$.  $V$ is referred to as the node set and $E$ is referred to as the edge set.  
The results in this paper allow for  a weighted graph. Let $w_{ij}$ be the weight of the edge $(i,j) \in E$; if $(i,j) \not \in E$, define $w_{ij} = 0$. If the graph is unweighted, then let $w_{ij} = 1$ for all $(i,j) \in E$.  Throughout this paper, the graph is undirected, $w_{ij} = w_{ji}$; for all pairs $i,j$.   Define the degree of node $i$ as 
$deg(i) = \sum_j w_{ij}$ and the volume of the graph as $vol(G) = \sum_i deg(i)$. If the graph is unweighted, $deg(i)$ is the number of connections to node $i$. To simplify notation, $i \in G$ is used synonymously  with $i \in V$.

\subsubsection{Markov chain on $G$}   
Denote $X(0), X(1), X(2) \dots  \in G$ as a Markov chain on the individuals from the social network $G$.  The transition matrix $P \in \R^{\gn \times \gn}$ is defined so that transition probabilities are proportional to edge weights,
\[P_{ij} = \pr\left( X(t+1) = j | X(t) = i\right) = \frac{w_{ij} }{ deg(i)}.\]
Let $|\lambda_1| \ge |\lambda_2| \ge \dots \ge |\lambda_N|$ denote the eigenvalues of $P$.   All eigenvalues of $P$ are less than or equal to one in absolute value (see e.g. Lemma 12.1 in \cite{levin2009markov}).   Because the edge weights are symmetric, $w_{ij} = w_{ji}$ for all $i,j$, the Markov chain is \textbf{reversible}.   If $|\lambda_2|<1$, then the stationary distribution $\pi: G \rightarrow \R$ is
\[\pi_j = \lim_{t \rightarrow \infty} \pr( X(t) = j |  X(0) = i)  = \frac{deg(j)}{vol(G)} \  \mbox{ for all $i,j \in G$}.\]

%If each RDS participant refers exactly one future participant, then the process could be indexed by a chain.  To allow each participant to have multiple referrals, the process should be indexed by a tree.

%The second eigenvector $f_2$ plays a critical role in the remainder of the paper.  Each of the leading eigenvectors represents a bottleneck in the referral process.
%In the previous example with two communities EAST and WEST, suppose that $P$ corresponds to the simple random walk. If EAST and WEST correspond to the most dominant partition in the network, then $f_2(i)$ and $f_2(j)$ will have the same sign (i.e. +/-) if and only if $i$ and $j$ belong to the same community. By looking at the signs, $f_2$ partitions the graph into EAST and WEST.  
%The Cheeger bound \citep{chung1997spectral} makes this concept rigorous and provides an argument for why spectral clustering can partition a graph \citep{von2007tutorial}.  

\subsubsection{Markov process on $G$ indexed by a tree}

%Another graph beyond the social network $G$ is needed to index the network driven sample.  While $G$ contains the entire population, the second graph contains the sampled nodes.  For example, under a Markov chain, this second graph contains the nodes $0, 1, 2, \dots, n$ and the edges $t-1 \rightarrow t$ for $t = 1, \dots, n$.  In this second graph, an edge corresponds to a referral in the sampling process.  Because some participants provide  multiple referrals and other participants provide zero referrals, this rest of the paper indexes the Markov process with a tree. 

Let $\T$ be a rooted tree--a connected graph with $n$ nodes, no cycles, and a vertex $0$.  The seed participant is vertex $0$ in $\T$ (cf Figure \ref{fig:graphSeq}).  Note that the node set of $G$ indexes the population and the node set of $\T$ indexes the sample.  To simplify notation, $\sigma \in \T$ is used synonymously with $\sigma$ belonging to the vertex set of $\T$.  For any node in the tree $\sigma \in \T$,  denote $\sigma' \in \T$ as the parent of $\sigma$ (the node one step closer to the root). Let $\D(\sigma) \subset  \T$ denote the set of $\sigma$ and all its descendants   in $\T$.   Denote the height of $\T$ as $h(\T)$; this is the number of rounds of sampling in the RDS, or the maximum graph distance in $\T$ from the root to any node.

A Markov process indexed by $\T$ is a set of random variables $\{X_\sigma : \sigma \in \T\}$ satisfying the Markov property 
\[\pr(X_\sigma | X_{\sigma'}, X_\tau : \tau \in \D(\sigma)^c) = \pr(X_\sigma | X_{\sigma'}).\]

The transition matrix $P \in [0,1]^{\gn \times \gn}$ describes these transition  probabilities,
\[\pr(X_\sigma = j | X_{\sigma'} = i) = P_{ij}, \mbox{ for } i,j \in G.\]
\cite{benjamini1994markov} called this process a \tpns.   Unless stated otherwise, it will be presumed throughout that under the \tpns,  $X_0$ is initialized from the stationary distribution of $P$.

%While $G$ contains the entire population, $\T$ indexes the sampled participants.  $\gn$ is the size of the total population and $\tn$ is the sample size.  For example, if $\tilde \T$ contains the nodes $0, 1, 2, \dots, n$ and the edges $t \rightarrow t+1$, then the $(\tilde \T, P)$-\textit{walk on} $G$ is a standard Markov chain, $X(0), X(1), \dots, X(n)$; the node $j \in \tilde \T$ corresponds to the $j$th person sampled, while the particular individual $X(j) \in G$ is random.  
%In the social network $G$, an edge represents friendship.  In the tree, an edge from $\tau \in \T$ to $\sigma \in \T$ represents that random individual $X_{\tau} \in G$ refers random individual $X_{\sigma} \in G$ in the \tpns.  

For example, if $\C$ is the chain graph, then the $(\C, P)\textit{-walk}$ on $G$ is a Markov chain on $G$, $X(0), X(1), X(2), \dots  \in G$.  One key property of the Markov model is that it allows for resampling.  Said another way, it ``samples with-replacement''  because it is possible for $X(i) = X(j)$ for $i \ne j$.  The same is true in the tree model.  In particular, it is possible for $X_\tau = X_\sigma$ for $\tau, \sigma  \in \T$ with $\tau \ne \sigma$.

\subsection{Measurements and estimators} \label{sec:measurements}
For each node $i \in G$, let $y(i) \in \R$ denote some characteristic of this node.  We wish to estimate the population average 
\[\mu_{true} = \frac{1}{\gn} \sum_{i \in G}  y(i).\]
In the motivating RDS example, $y(i) = 1$ denotes that $i \in G$ is HIV+, $y(i) = 0$ denotes that $i \in G$ is HIV-, and $\mu_{true}$ is the proportion of the population that is HIV+.  We estimate $\mu_{true}$ with observations
\[Y_\tau = y(X_\tau) \ \mbox{ for } \tau \in \T,\] 
where $X_\tau$ is a \tpns.  Denote
\[\mu = \E_{RDS}(Y_0) = \sum_i y(i) \pi_i,\]
where the subscript  $_{RDS}$ denotes that the expectation is computed with the \tpns.  In general, $\mu \ne \mu_{true}$.
The sample average,
\begin{equation}\label{eq:avg}
\hat \mu = \frac{1}{n} \sum_{\tau \in \T} Y_\tau
\end{equation}
is an unbiased estimate of $\mu$.

With $\pi_i = deg(i)/vol(G)$,
the inverse probability weighted estimator (IPW),
\[\hat \mu_{IPW} 
= \frac{1}{n} \sum_{\tau \in \T} \frac{Y_\tau }{\pi_{X_\tau}N}
= \frac{vol(G)}{N} \frac{1}{n}   \sum_{\tau \in \T} \frac{Y_\tau }{deg(X_\tau)},\]
is an unbiased estimator for $\mu_{true}$. The results in this paper can be applied to $\hat \mu_{IPW}$ via a transformation that is described in the next remark.  Computing the IPW estimator requires $vol(G)$ or the average node degree $vol(G)/N$.  This is typically not available in practice.  
When the sampling weights can be identified up to a constant of proportionality (i.e. $\pi_i \propto deg(i)$), estimating $vol(G)/N$ with the harmonic mean of the observed node degrees,
\[H = \left(n^{-1} \sum_{\tau \in \T} 1/ deg(X_{\tau})\right)^{-1},\]  
leads the Hajek or Volz-Heckathorn estimator \citep{volz2008probability},
\[\hat \mu_{VH} = H \frac{1}{n} \sum_{\tau \in \T} \frac{Y_\tau}{deg(X_\tau)}.\]

\begin{remark} \label{remark:vh}
Define a new node feature 
\[y^\pi(i) = \frac{y(i)}{\pi_iN}\] 
and new node measurements $Y^\pi_\tau = y^\pi(X_\tau)$.  The sample average of the $Y^\pi_\tau$'s is exactly the IPW estimator using the non-transformed $Y_\tau$'s.  Because of this simple transformation, the theorems below that study $\hat \mu$ can also study $\hat \mu_{IPW}$ by substituting $y^\pi$ for $y$.
\end{remark}

Define $W_1, \dots, W_n \in G$ as independent random samples with $\pr(W_i = j) = \pi_j$.  Define
\begin{equation} 
\label{eq:varrs} Var_{\pi} (\hat \mu) = Var\left(\frac{1}{n} \sum_{i = 1}^n y(W_i)\right). 
\end{equation}
Define the \textbf{design effect} of the \tp as
\begin{equation}\label{eq:de}
DE(\hat \mu) = \frac{Var_{RDS}(\hat \mu) }{  Var_{\pi}(\hat \mu)}.
\end{equation}
The standard definition of $DE$ contains the variance under simple random sampling (SRS) in the denominator.  For simplicity, the $DE$ in this paper contains $Var_\pi$ in the denominator.  The key difficulty of comparing SRS to the \tp is that SRS is without-replacement.  
%Note that 
%In our case, 
%It is unfair to compare the \tp to simple random sampling because simple random sampling is without-replacement.  I
%Define $Var_u(\hat \mu)$ to be . 
Instead of SRS, the denominator in Equation \eqref{eq:de} could be replaced by the variance under uniform sampling (with-replacement) and this would only change the $DE$  by a constant factor.  This is because $G$ and $\pi$ do not change with $n$.

The standard O-notation is used below.  In particular, $h(\tn) = o(g(\tn))$ means that $h(\tn)/g(\tn) \rightarrow 0$ as $n \rightarrow \infty$ and 
$h(\tn) =O(g(\tn))$ means that $h(\tn) \le M g(\tn)$ for all $\tn$, for some constant $M$. 

\section{The variance under RDS} \label{sec:var}

The key result of this section, Theorem \ref{thm:var}, expresses $Var_{RDS}(\hat \mu)$ as a function of the eigen-properties of $P$.  
The following lemma from \cite{levin2009markov} provides the eigendecomposition of the matrix $P$. 
\begin{lemma} (Lemma 12.2 in \cite{levin2009markov}) \label{lem:tmp}
Let $P$ be a reversible Markov transition matrix on the nodes in $G$ with respect to the stationary distribution $\pi$.  The eigenvectors of $P$, denoted as $f_1, \dots, f_{\gn}$, are real valued functions of the nodes $i \in G$ and orthonormal with respect to the inner product 
\begin{equation} \label{def:inner}
\langle f_a, f_b \rangle_\pi = \sum_{i \in G} f_a(i) f_b(i) \pi_i.
\end{equation}
If $\lambda$ is an eigenvalue of $P$, then $|\lambda|\le 1$.  The eigenfunction $f_1$ corresponding to the eigenvalue $1$ can be taken to be the constant vector $\textbf{1}$.
\end{lemma}

All of the statements in this section are conditional on the tree.  This tree appears in the formula for the variance through the functional $\G$, defined as follows.

%In the various forms of network sampling, $\T$ is observed and can be used for the data analysis.  In this section, condition on the tree.  If the tree $\T$ change from one experiment to the next.   So, The next section studies how $\G$ depends on only a simple summary of $\T$ and this allows for the study of random trees.
%
%
%Theorem \ref{thm:var} below shows that the variance of $\hat \mu$ changes with the tree.  A similar phenomenon happens in linear regression, where (1)  the design matrix controls the covariance of the estimator and (2) the design matrix may change if the experiment were repeated.  Because of this, the standard estimators of the covariance in linear regression \textit{condition} on the design. One argument for conditioning is formalized by the conditionality principal in \citep{birnbaum1962foundations}.  
%In an analogous fashion, Theorem \ref{thm:var} conditions on the tree $\T$;  

%The corollary give upper and lower bounds that depend on $\G(\lambda_2)$, where $\lambda_2$ is the second eigenvalue of $P$ and $\G$ is defined as follows.

\begin{definition} \label{def:probgen}
Select two nodes $I, J \in \T$ uniformly and independently.  Define  $D = d(I,J)$ to be the graph distance between $I$ and $J$ in $\T$.  Define $\G$ as the probability generating function for $D$,
\[\G(\lambda) = \E( \lambda^{D}).\]
\end{definition}
%Because $\T$ is observed, the function $\G$ can be computed (see Figure \ref{fig:bothTrees}).  
%In many studies there are multiple seed nodes.  In these cases, $\T$ is a ``forest'' and $\G$ can be computed by setting $d(I,J) = \infty$ if $I$ and $J$ are in different connected components of $\T$.  

Note that because the tree $\T$ is observed, the function $\G$ can be computed in practice.  Figure \ref{fig:bothTrees} gives an illustration of $n \G(\lambda)$ for $\lambda \in [0,1]$, where $n$ is the number of nodes in $\T$.

\begin{theorem} \label{thm:var}
Suppose that the Markov transition matrix $P$ is reversible with respect to $\pi$ and that the second eigenvalue of $P$ is less than one in absolute value, then 
\begin{equation} \label{eq:var}
Var_{RDS}(\hat \mu) \ = \  \sum_{\ell=2}^{\gn} \langle y, f_\ell \rangle_\pi^2 \G(\lambda_\ell),  
\end{equation}
where the subscript $_{RDS}$ denotes that data have been collected through a \tpns, $\hat \mu$ is defined in Equation \eqref{eq:avg}, $\langle \cdot, \cdot \rangle_\pi$ is defined in Equation \eqref{def:inner}, $f_1, \dots, f_\gn: G \rightarrow \R$  are the eigenvectors of $P$ corresponding to eigenvalues $\lambda_1 > |\lambda_2| \ge  \dots \ge |\lambda_\gn|$, and $\G$ is defined in Definition \ref{def:probgen}. 
\end{theorem}
In previous research, \cite{verdery2013network} and \cite{acrds}  prove this theorem for the special case that $\T$ is a chain.   
%Theorem \ref{thm:var} finds a closed form expression for the variance of $\hat \mu$ for any referral tree $\T$ by decomposing $y$ with the eigenbasis for the Markov transition matrix $P$.  In an asymptotic setting where the number of samples is growing, the coefficient $\langle y, f_\ell \rangle_\pi^2$ and the eigenvalues $\lambda_\ell$ remain unchanged; it is the function $\G$ that changes with $n$.  
The first step to prove Theorem \ref{thm:var} is to show that if $d(\sigma, \tau) = t$, then by the reversibility of $P$,
\[(X_\sigma, X_\tau) \stackrel{d}{=} (X(0), X(t)),\]
where $X(0), \dots, X(t) \in G$ is a Markov chain with the same transition matrix $P$.  Then, expanding $y$ in the eigenbasis from Lemma \ref{lem:tmp}, 
\begin{equation}\label{eq:cov}
Cov_{RDS}(Y_\sigma, Y_{\tau})= \sum_{\ell=2}^{\gn} \lambda_\ell^{d(\sigma, \tau)}  \langle y, f_\ell \rangle_\pi^2.
\end{equation}
Averaging over $\sigma, \tau$ and exchanging summations yields $\G$ and the final result.  Section \ref{varProof} in the appendix contains a full proof.

\begin{figure}[h] %  figure placement: here, top, bottom, or page
   \centering
   \includegraphics[width=6.5in]{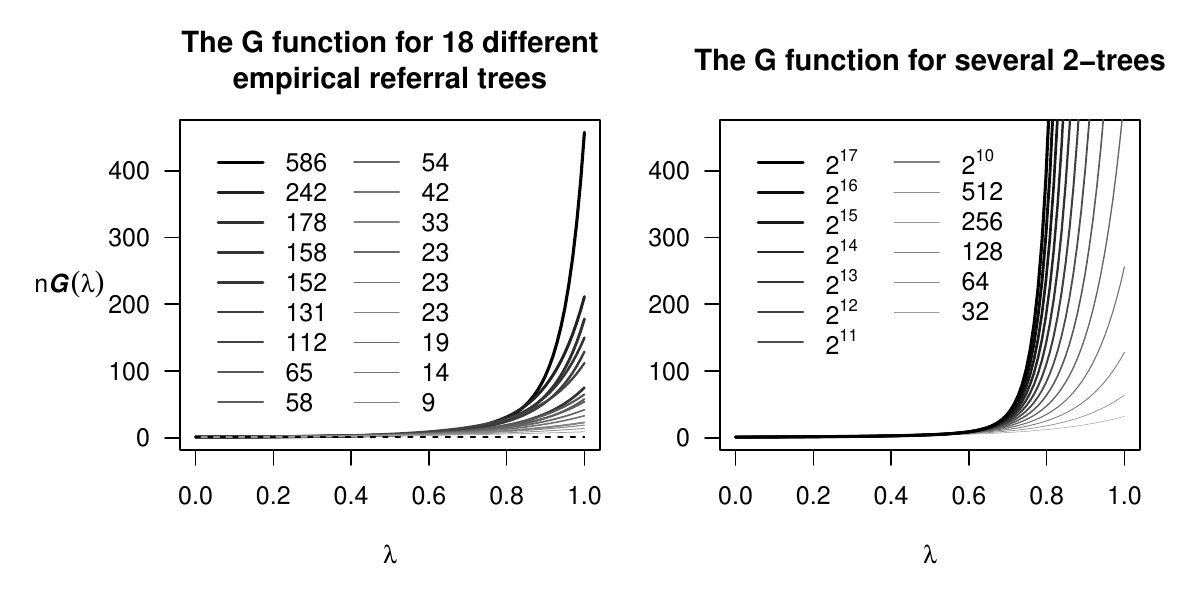} 
   \caption{Each line corresponds to a referral tree.  The vertical axis gives $n \G(\lambda)$.  The legend gives the number of nodes in each tree.   In the left panel, there are eighteen different referral trees from published RDS studies.  The tree of 586 comes from a study of drug users in New York City \citep{abdul2006effectiveness}.  The tree of 112 comes from a study of injection drug users in Connecticut \citep{heckathorn1997respondent}.  The trees of 14, 19, 23, 23,  65,  and 152  come from a study of men who have sex with men in Higuey, Dominican Republic \citep{gile2015diagnostics}.  The remaining ten trees come from a study of 25 villages in rural Uganda \citep{mccreesh2012evaluation}.  
    In the right panel, each line represents a $2$-tree, where each node creates two nodes in the next wave.  
%   The left panel gives   for 18 empirical .  The right panel presents the same results for several $2$-trees.   On the horizontal axis, $\lambda$ represents the strength of a potential referral bottleneck.   
%    All lines are upward sloping.  
%    Moreover, larger trees are more sensitive to larger values of $\lambda$.  
The results of the next section are foreshadowed by the critical threshold at $\lambda = 1/\sqrt{2} \approx .7$.
%As expected, all of the lines are upward sloping, indicating that stronger bottlenecks produce larger design effects.  Moreover, the larger trees are more sensitive to the bottleneck strength.
    }
   \label{fig:bothTrees}
\end{figure}

\begin{remark}\label{remark:seedbias}
Using Remark \ref{remark:vh}, Theorem \ref{thm:var} also gives the variance for $\hat \mu_{IPW}$.  This theorem  presumes that $X_0$ (i.e. the seed node) is sampled from the stationary distribution.  Under this assumption, $\hat \mu_{IPW}$ is unbiased for finite samples.  However, conditionally on the seed node, $\hat \mu_{IPW}$ and $\hat \mu_{VH}$ are biased \citep{gile2010respondent}.  The law of total variance shows how $Var_{RDS}(\hat \mu_{IPW})$ includes the seed-bias, defined as $bias(\hat \mu_{IPW}, X_0) = \E_{RDS}(\hat \mu_{IPW}|X_0) - \mu_{true},$
\begin{eqnarray*}
Var_{RDS}(\hat \mu_{IPW}) 
%= E\left(Var(\hat \mu_{IPW}|X_0)\right) + Var\left(E(\hat \mu_{IPW}|X_0)\right)\\
= \E_\pi\left(Var_{RDS}(\hat \mu_{IPW}|X_0)\right) + \E_\pi\left(bias(\hat \mu_{IPW}, X_0)\right)^2,
\end{eqnarray*}
where $\E_\pi$ is the expectation with respect to $X_0$ having distribution $\pi$.  
\end{remark}

The law of total variance shows that that conditioning on the seed node decreases the variance by the squared seed-bias.

\section{The asymptotic behavior of the design effect} \label{sec:critical}
The eigen-properties of $P$ have been extensively studied in 
the literature on spectral graph theory and spectral clustering \citep{chung1997spectral,von2007tutorial}.
Cheeger's inquality shows that if $\lambda_2$ is close to one, then there are clusters or communities in the graph.  
For RDS, this creates a ``referral bottleneck'' where the referral process has difficulty mixing between the two communities. 
For example, if $\lambda_2=1$, then the social network is disconnected; this represents an extreme  bottleneck, where the referral process will never cross the divide.  This section shows that the asymptotic behavior of $DE$ depends upon the relationship between $\lambda_2$ and the growth rate of the referral tree $\T$.

To study how $Var_{RDS}$ and $DE$ behave as the sample size increases, it is necessary to describe how the referral tree $\T$ grows.  
Theorem \ref{thm:critical} grows a random Galton-Watson tree.  A Galton-Watson tree is initialized with a single root node and is parameterized by its \textbf{offspring distribution}.  Starting with the root node and iterating through all future generations,  each node generates a random number of offspring, drawn from the offspring distribution.  The number of offspring produced by each node is independent across nodes.  This process is highly studied with several well known results (e.g. \cite{athreya1972branching}). 

Let $\xi$ be a generic draw from the offspring distribution  and denote $\E \xi =m$.    To have a positive probability that the tree generates an infinite number of nodes, the results below require that $m>1$.  Denote $\T_h$ as the sub-tree of $\T$ that includes all nodes within distance $h$ of the root.

\begin{theorem} \label{thm:critical}
Suppose $\T$ is a random Galton-Watson tree.  Let $\xi$ be a single draw for the offspring distribution with $m = \E(\xi) >1$ and $\E(\xi^4) < \infty$. Condition on the survival of the Galton-Watson process.  
Define $\T_h$ as the node induced subgraph of $\T$ that contains all nodes $\tau \in \T$  within distance $h$ from the root node.  
Let $P$ be a Markov transition matrix on $G$ that is reversible with respect to its stationary distribution $\pi$. 
Let $\hat \mu_h$ be constructed with the samples from a $(\T_h,P)\textit{-walk on } G$.
If $Var_\pi Y_0 >0$, $\langle y,  f_2 \rangle_\pi^2 > 0$, and $\lambda_2>0$, then  
\begin{equation} \label{eq:critical} \arraycolsep=1.4pt\def\arraystretch{1.4}
%%\G_{h}(\lambda_2) \le 
%\mbox{Design Effect} 
DE(\hat \mu_h) \asymp
\left\{\begin{array}{ll}
 c & \mbox{ if } m \le \beta \\
% n^{-1}& \mbox{ if } m = \beta \\
n^{1-\alpha} & \mbox{ if } m > \beta,
\end{array}\right.
\end{equation}
where $DE$ is defined in Equation \eqref{eq:de} conditionally on $\T$, $\asymp$ is equality up to $(\log n)^2$ terms, $\beta = \lambda_2^{-2}$, and $\alpha = \log_m \lambda_2^{-2}$.  
\end{theorem}

The proof of this result has four pieces, divided into four subsections of Section \ref{sec:ProofCritical} in the appendix. 
Section \ref{sec:PrelimCritical} shows that $DE$ behaves asymptotically similar to $n\G(\lambda_2)$.   Then, Subsection \ref{sec:lowerBound} gives a lower bound for $\G(\lambda_2)$ that depends only on the growth rate of the tree $\T$.  Subsection  \ref{sec:upperBound} gives an upper bound for $\G(\lambda_2)$ that requires a ``balanced assumption'' on $\T$.  These three subsections do not require that $\T$ comes from the Galton-Watson distribution.  Then, in Section \ref{sec:GWP} the Kesten-Stigum Theorem shows that when $\T$ comes from the Galton-Watson distribution, it grows at rate $m$ (satisfying the lower bounds in Section \ref{sec:lowerBound}).  
Then, Lemma \ref{thm:gwp}
applies the $L^p$ maximal inequality for martingales to  the  Galton-Watson martingale to show that Galton-Watson trees with $\E \xi^4 < \infty$ satisfy the ``balanced assumption.''

The assumption that $\E(\xi^4) < \infty$ is a strong assumption in the literature on the Galton-Watson process.  However, there are two important points.  First, in the context of RDS, the offspring distribution is typically bounded by three or five. As such, this condition is certainly satisfied. Second, the finite fourth moment is only needed for the upper bound.  So, if the fourth moment were infinite, then the $DE$ could be much larger.

To see why there is a critical threshold,
note that $Var_{RDS}(\hat \mu)$ is the average of the covariances $Cov_{RDS}(Y_\sigma, Y_\tau)$.  From Equation \eqref{eq:cov} each covariance term decays exponentially, $O(\lambda_2^{d(\sigma, \tau)})$,  where $d(\sigma, \tau)$ is the graph distance between $\sigma$ and $\tau$ in $\T$.  However, these graph distances grow logarithmically; when $m>1$, $d(\sigma, \tau) = O(\log_m n)$.  For example, if $\T$ is a complete $m$-tree with $n$ nodes, $h(\T) \le \log_m n$ implies $d(\sigma, \tau) \le 2 h(\T) \le 2 \log_m n$.  Using these bounds,
\[\lambda_2^{d(\sigma, \tau)} \ge \lambda_2^{2 \log_m n} = n^{2 \log_m \lambda_2}.\]
The critical threshold comes from a competition between (i) the logarithmically expanding distances and (ii) the exponentially contracting covariances.
Above the critical threshold, the upper bound in the appendix confirms that $n^{2 \log_m \lambda_2}$ is the rate of $Var_{RDS}(\hat \mu)$.  Below the critical threshold, $Var_{RDS}(\hat \mu)$ is controlled by the terms $\sigma = \tau$ and the variance converges at the standard $O(n^{-1})$ rate.   For more details, see Section \ref{sec:ProofCritical} in the appendix.

\section{The gap between sampling with and without-replacement} \label{sec:replacement}

Define the number of repeated pairs as 
\[R_n = |\{\sigma, \tau \in \T| \tau \ne \sigma, X_\tau = X_\sigma\}|.\]
This section studies $\E_{RDS}(R_n)$ as $n$ and $N$ grow in tandem. Because $R_n$ counts \textit{pairs} of repeats, $\E(R_n)$ could grow at rate $n^2$.  Proposition \ref{prop:lowerRn} and Theorem \ref{thm:upperRn} show that if $n = o(\sqrt{N})$ and some additional  assumptions, then $\E_{RDS}(R_n) \asymp n$.  In particular, this shows that the rate of resampling does not depend on $\lambda_2$.  

\begin{prop} \label{prop:lowerRn}
Under the \tpns, suppose that $G$ is undirected and $P$ is a simple random walk.  If $deg(i) <D$ for all nodes $ i \in G$, then
\[\E(R_n) \ge n/D. \]
%for some constant $c'$. 
\end{prop}
The proof is based on the fact that if $X_\sigma = i$, the probability of transitioning back to the state of $X_{\sigma'}$ is $1/deg(i) \ge 1/D$.  The full proof is contained in Section \ref{app:replacement} of the appendix.

As  Proposition \ref{prop:lowerRn} shows, the \tp can have several repeated samples.  However, this alone does not prevent the variance from decaying at rate $1/n$; the decay of the variance is determined by the critical threshold, $m > \lambda_2^{-2}$. 
 The next result gives a matching upper bound for $\E(R_n)$.  This shows that the rate of $\E(R_n)$ does not depend on the critical threshold.

\begin{theorem} \label{thm:upperRn}
Consider a sequence of samples  $\{X_\tau: \tau \in \T_n\}$ that are sampled from a $(\T_n,P_N)$ \textit{-walk on } $G_N$, where $n$ and $N$ are both growing. Suppose that the sequence $\T_n$ satisfies the conditions of Theorem \ref{thm:upperBound}; that is, there is a balanced infinite tree $\T$ that grows at rate $m$ and $\T_n$ is a sequence of subtrees that successively add one generation at a time.

If (1) the stationary distribution is bounded, $\pi_i \le c/N$ for all $i$ and all $N$; (2)  the number of eigenvalues $\lambda_\ell$ that exceed the critical threshold $1/\sqrt{m}$ is bounded by $k$ for all $N$;  and (3) $n = o(\sqrt{N})$, then
\[\E(R_n) = O((\log n) n).\]
\end{theorem}
Notice that condition (1) is implied by the bounded degree assumption in Proposition \ref{prop:lowerRn}.  Importantly, the rate of this upper bound does not depend on $\lambda_2$.    So, under the conditions of these results, $\lambda_2$ and the critical threshold do not effect the rate of $\E(R_n)$.

The key to proving Theorem \ref{thm:upperRn} is the relationship between the trace of a matrix and its eigenvalues.  First, notice that 
\begin{equation}\label{eq:ern}
\E(R_n) =  \sum_{\sigma \ne \tau} \pr(X_\sigma = X_\tau).
\end{equation}
Let $tr(P)$ denote the trace of $P$.
\[\pr(X_\sigma = X_\tau) 
= \sum_{i \in G} \pi_i \pr( X_\tau = i|X_\sigma = i) = \sum_{i \in G} \pi_i P_{ii}^{d(\sigma, \tau)} \le c N^{-1} tr(P^{d(\sigma, \tau)}) = c N^{-1} \sum_\ell \lambda_\ell^{d(\sigma, \tau)}\]

To bound $\E(R_n)$, exchange the summation over $\sigma \ne \tau$ from  Equation \eqref{eq:ern} with the summation over $\ell$ in the line above.  Each term in the resulting summation can be expressed with  $\G$ functions and bounded by Theorem \ref{thm:upperBound}.  The full proof is contained in Section \ref{app:replacement} of the appendix.

\subsection{Comparison to a more realistic model with simulation} \label{sec:replacementSim}
For mathematical tractability, the theorems above make two simplifications.  First, the theorems use the \tpns, which samples with-replacement.    Second, the theorems study the IPW estimator.   
%In practice, the sampling is done without-replacement and the scaling constant in the IPW estimator is unknown. 
The simulations in this section (and in the rest of the paper) use a more realistic setting.  First, the simulated samples are collected without-replacement.  Second, the simulations study the Volz-Heckathorn estimator.
 These simulation results find that the Markov model with the IPW estimator is a good approximation to the more realistic model, so long as the number of sampled nodes is much smaller than the population size, as predicted by Theorem \ref{thm:upperRn}.  
 
The simulations are performed on networks simulated from the Stochastic Blockmodel.   The ten panels in Figure \ref{fig:replacement} correspond to ten different model settings.  Each of the ten models has $N =$10k nodes, equally balanced between group zero and group one.  The probability of a  connection between two nodes in different blocks is $r$ and the probability of connection between two nodes in the same block is $p$.  Figure \ref{fig:replacement} parameterizes this model via (1)  the expected degree $(p+r)N/2$ and (2) the second eigenvalue of $\mathscr{P} = \E(D)^{-1} \E(A)$, 
\begin{equation}\label{eq:sbmlambda2}
\lambda_2(\mathscr{P}) = \frac{p-r}{p+r},
\end{equation}
where expectations are under the Stochastic Blockmodel (cf example on page 1893 of \cite{rohe2012sp}). In group zero, $y_i = 0$ and in group one, $y_i =1$.  The horizontal axis in each plot represents the sample size; the vertical axis represents the design effect (as estimated via simulation).   The five columns of plots correspond to five different values of $\lambda_2(\mathscr{P})$.

\begin{figure}[h] %  figure placement: here, top, bottom, or page
   \centering
   \includegraphics[width=6in]{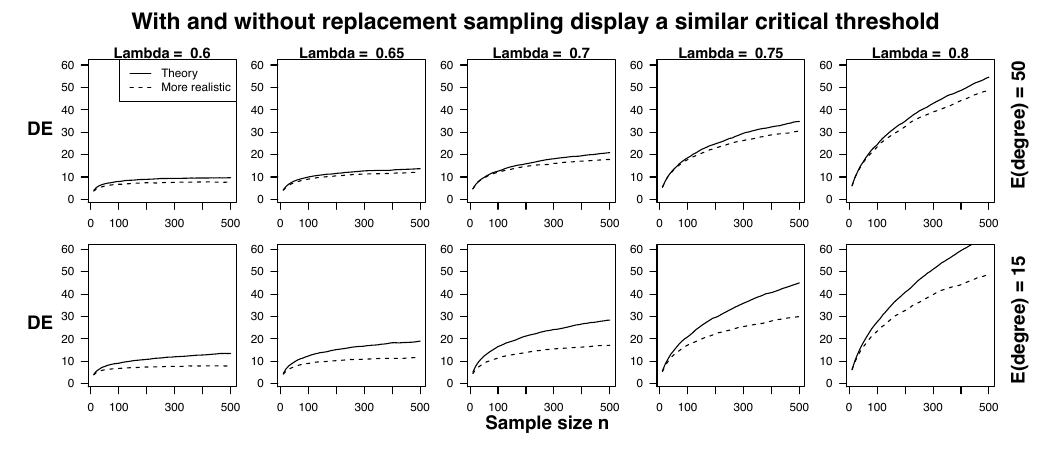} 
   \caption{In all figures, $m =2$.  Each column of panels corresponds to a different value of $\lambda_2$, from left to right, $\lambda_2 \in (.6, .65, .7, .75, .8)$.  In the panels on the left, the lines are roughly flat. In the panels on the right, the lines are quickly increasing.  This shows that the \tp and the more realistic model have a critical threshold somewhere between $\lambda_2(\mathscr{P})= .6$ and $\lambda_2(\mathscr{P})=.8$.}
   \label{fig:replacement}
\end{figure}

To simulate from the more realistic model, 
the simulation first generates $\T$ as a Galton-Watson tree with offspring distribution $1 + $ Binomial$(2,1/2)$.  A tree is grown until it reaches 2000 nodes; while only 500 samples are kept, it will become clear why $\T$ must be initialized to be larger than 500.  This tree is seeded with a participant selected from the stationary distribution.  
Then, each participant randomly selects their referrals from their ``viable'' friend list
 without-replacement; a friend is viable if it has not yet appeared in the sample.
One participant at a time makes all of their referrals, iterating through the tree in the fashion of a breadth first search.  A difficulty arises if  $\sigma \in \T$ should produce three referrals, but $X_\sigma$ does not have that many viable friends.  When this happens in the simulation, all viable friends are referred and the remaining descendants in $\T$ are removed; this happens infrequently in the simulation. Once this process samples 500 nodes, the remaining nodes in $\T$ are pruned.  This pruned tree is then used to run the \tpns.  For each of the ten networks, this process is simulated 1000 times.  The sample variance across these 1000 samples is divided by the variance of uniform with-replacement sampling, $(4n)^{-1}$. 

Because the trees are simulated to have $m=2$, Theorem \ref{thm:critical} suggests that the design effect grows when $\lambda_2$ exceeds $1/\sqrt{2} \approx .7$.  In the left most plots, the solid lines are roughly flat. In the right most plots, the solid lines are quickly increasing.  This shows that the \tp has a critical threshold somewhere between .6 and .8; this is consistent with the theory.  Similarly, the dashed lines are roughly flat in the left plots and quickly increasing in the right plots. Under these simulation settings, the more realistic model mimics the critical threshold behavior identified in Theorem \ref{thm:critical}. 

In the first row of plots, each node has an expected degree of fifty.  In the second row of plots, each node has an expected degree of fifteen.  In the top row, the solid and dashed lines are close because there are fewer repeated samples.  In the bottom row, the lines for the sparse graphs are not as close.  However, both rows display the same qualitative behavior (flat when $\lambda_2 = .6$ and increasing when $\lambda_2=.8$).

\section{Reinterpreting the results of \cite{goel2010assessing} with Theorem \ref{thm:critical}} \label{sec:reinterpreting}
One of the most highly cited bootstrap procedures in the previous literature was proposed in \cite{salg} and is often referred to as the Salganik bootstrap.  Later, \cite{goel2010assessing} showed in  simulation experiments  that this procedure produces ``misleadingly narrow'' confidence intervals.  This section reinterprets those simulation results using Theorems \ref{thm:var} and \ref{thm:critical} above.  This reinterpretation motivates an alternative bootstrap procedure which is explored in the next section.

In the simulation study, \cite{goel2010assessing} used several different graphs $G$ that were collected in previous empirical social network research.  In each of several experiments, $y$ is a demographic measurement such as race or gender.  Given $G$ and $y$, \cite{goel2010assessing} simulated the respondent-driven sample with a \tpns, where $\T$ is a Galton-Watson tree with $m = 1.5$.  After collecting a sample of $n=500$, \cite{goel2010assessing} constructed a bootstrapped confidence interval with the Salganik bootstrap \citep{salg}. To resample the observed individuals, the Salganik bootstrap constructs a Markov transition matrix $\hat P_o \in \R^{n\times n}$ on the observed individuals as follows:

\begin{quote}
[D]ivide the sample members into two sets based on how they were recruited: people recruited by someone in group $A$ (which we will call $A_{rec}$) and people recruited by someone in group $B$ (which we will call $B_{rec}$). For example, $A_{rec}$ could be the set of all sample members who were recruited by someone with HIV. $\dots$ [B]ased on the group membership [of the current state], we draw with-replacement from either $A_{rec}$ or $B_{rec}$. \citep{salg}
\end{quote}  
The fundamental problem with the Salganik bootstrap is that each bootstrap sample is a $(\C, \hat P_o)\textit{-walk}$ on the observed individuals, where $\C$ is a chain graph.  By using $\C$ instead of $\T$, \textit{the Salganik bootstrap resampling distribution is a Markov chain, not a ``Markov tree.''}

In the simulation results of \cite{goel2010assessing}, the bootstrap has particularly poor coverage on a subset of the features.  
These features are correlated with the underlying social network.  In particular, if there is an eigen-pair $(\lambda_\ell, f_\ell)$ of $P$ where $\langle y, f_\ell \rangle_\pi^2$ is large and $1.5 > 1/\lambda_\ell^2$, then the \tp exceeds the critical threshold, while the $(\C, \hat P_o)\textit{-walk}$ does not.  
If the original sample \tp exceeds the critical threshold, then estimates derived from this sample will be highly variable.  However, because the $(\C, \hat P_o)\textit{-walk}$ resamples with a chain graph $\C$, it has $m=1$.  As such, the $(\C,\hat P_o)\textit{-walk}$ will never exceed the critical threshold.  The confidence intervals from the Salganik bootstrap will contract at rate $O(n^{-1/2})$, while the true uncertainty is decaying at a slower rate.  This leads to confidence intervals which are too narrow.

\section{Bootstrap resampling with $\T$} \label{sec:bootstrap}

To allow for the bootstrap distribution to exceed the critical threshold, this section proposes \atbns.  The Salganik bootstrap will be referred to as \achbns.  The \textsc{a-} prefix stands for \textit{assisted}, because they are both assisted by some node feature to create a Markov transition matrix.  In the \achbns, the construction of the matrix $\hat P_o$ is assisted by the outcome of interest $y$ (via the sets $A_{rec}, B_{rec}$).  
%Because $\hat P_o$ depends upon the outcome of interest, a new set of bootstrap samples must be simulated for each feature of interest.  
The \atb  also constructs a Markov transition matrix on the observed individuals, $\hat P \in \R^{n \times n}$, and the construction of this matrix is assisted by some node features.  However, unlike \achbns, \atb does not require that $\hat P$ is constructed from the same variable as the outcome of interest $y$.  As described in the previous section, the \achb directly samples from the $(\C, \hat P_o)\textit{-walk}$.  Similarly, the \atb   directly samples from the $(\T,\hat P)\textit{-walk}$, where the construction of $\hat P$ is described in the next subsection. \texttt{R} code for \atb is available at \url{https://github.com/karlrohe/mRDS}.

% and the \atb will uses these characteristics to construct an alternative $\hat P$. 

Over the course of this research, \cite{baraff2016estimating} proposed another bootstrap procedure which also uses $\T$ to perform the resampling.  This procedure will be referred to as \utbns. The \textsc{u-} prefix stands for \textit{unassisted} because it does not require any node features to construct its Markov transition matrix.  In particular, the \utb resamples from the $(\T, \hat P_u)\textit{-walk}$, where $\hat P_u \in R^{n \times n}$ is defined as follows:

\begin{quote}
\dots the initial step is to resample with-replacement from the seeds of the trees. Next, from each of those seeds, we resample with-replacement from their recruits, creating the second level of the bootstrap sample trees. From each of these new recruits, we then resample with-replacement from their recruits to create a third level. This process continues iteratively until no further recruits are available. \citep{baraff2016estimating}  
\end{quote}
To define $\hat P_u$ in the notation of this paper, let $A_\T \in \{0,1\}^{n \times n}$  be the (asymmetric) adjacency matrix of the directed graph $\T$.  So, for $\sigma \in \T$ with $\sigma \ne 0$, $[A_\T]_{\sigma', \sigma} = 1$.  All other elements of $A_\T$ are zero. Define $D_\T$ as a diagonal matrix containing the number of referrals from $\sigma$ in element $\sigma,\sigma$; $[D_\T]_{\sigma, \sigma} = \sum_\tau A_{\sigma,\tau}$.  Note that if $\sigma \in \T$ is a leaf node, then $[D_\T]_{\sigma, \sigma}=0$. 
The Markov transition matrix is $\hat P_u = D_\T^{-1} A_\T$, where $0/0$ is defined to be zero and the process terminates upon reaching a leaf node.   This $\hat P_u$ is neither irreducible nor reversible.

\subsection{The \atb procedure}

This subsection describes the construction of $\hat P$ used in the \atbns.  Presume that every node in $G$ belongs to a class, $z: V \rightarrow \{1, \dots, K\}$, and $z(i)$ is observed if node $i$ is sampled.  These variables could denote some demographic characteristics or HIV status.  The variables $\{z(X_\tau): \tau \in \T\}$ assist the estimation of the Markov transition matrix on the $n$ individuals  in the original sample $X_\tau$.  

All probability statements in this section are conditional on the original sample.  So, to temporarily conceal the randomness of the original sample, denote the observed individuals with lower-case letters, $\{x_\tau: \tau \in \T\}$.  Recall that for any $\sigma \in \T$ with $\sigma \ne0$, the parent node of $\sigma$ is denoted as node $\sigma' \in \T$.  Denote $N(u) =  \sum_{\sigma} \textbf{1}\{z(x_{\sigma})= u\}$ as the number of nodes in class $u$.  
Define $\hat A: \{1, \dots, K\}^2 \rightarrow \R$ to count the number of transitions between node types; for $u,v \in \{1, \dots, K\}$,
\begin{equation}\label{tildep}
\hat A(u,v) =  \sum_{\sigma \ne 0} \textbf{1}\{z(x_{\sigma'})= u, z(x_{\sigma})= v\}.
\end{equation}
Denote $\hat D(u)$ as the number of samples in class $u$ that make a referral, $\hat D(u)= \sum_v \hat A(u,v)$. 

If $X_0^*$ and $X_1^*$ represent one step of \utbns,  then $X_0^*$ and $X_1^*$ take values in the set of originally sampled individuals $ \{x_\tau : \tau \in \T\}$.  
%; this is $.  
For any $x_\sigma$ and $x_\tau$  in the original sample, define $u = z(x_\sigma)$ and $v = z(x_\tau)$. Then, the probability of a transition from $x(\sigma)$ to $x(\tau)$  in \utb 
% , define $\hat P$ such that
is defined to be
\begin{equation}\label{eq:Phat}
\hat P_{x_\sigma, x_\tau} = \pr \left(X_{1}^* = x_\tau | X_{0}^*= x_\sigma\right)= \frac{\hat A(u,v)}{\hat D(u)} \frac{1}{N(v)}.
\end{equation}
This is equivalent to first taking a Markov transition from $z(x_\sigma)$ to some other node type $v$
%sampling the type $v$ as a Markov transition on $\hat A$ from $z(x_\sigma)$ 
and then choosing an individual uniformly from the set of $N(v)$-many individuals of this type.  
Using the matrix $\hat P$, the resampling distribution of \atb is a $(\T,\hat P)\textit{-walk}$ on $\{x_\tau: \tau \in \T\}$.  Denote a resample as $\{X_\tau^*: \tau \in \T\}$; using these samples, construct $\hat \mu^*$ using $\{y(X_\tau^*): \tau \in \T\}$ and any other measured features on the originally sampled individuals $\{x_\tau: \tau \in \T\}$ (e.g. their degree in $G$).  

To sample the seed node(s), \atb first samples a ``mother node'' uniformly at random from the original sample.  Then, this mother node refers all of the seed nodes in the bootstrap sample. 
%The value of $y$ for this mother node is not used in the computation of $\hat \mu^*$.  
The mother node simulates the fact that some group is responsible for finding the seed nodes and this group is likely to constrained in their ability to select seeds.  Including the mother node in the resampling increases the dependence of the sample and thus the variability of $\hat \mu^*$.  

The key ideas of \atb can also be used to perform sample size calculations. 
%For practitioners wishing to perform power calculations, one 
%could use \textit{a priori} guesses inside of Equation \eqref{eq:Phat}.  
To do this, one must guess (i) $K$, (ii) for each $u,v \in 1, \dots K$, the probability that someone in class $u$ refers someone in class $v$,   (iii) the proportion of  individuals that belong to each class, (iv) the values of $y$ within each class, and (v) the topological structure of $\T$.  
%To perform both \atb and these power calculations, there is 
\texttt{R} code for this is available at \url{https://github.com/karlrohe/mRDS}.

% and the proportion of individuals in each class which are HIV+ 
%
%In order to aide faster and easier estimates, for studies that use \atbns, it is requested that papers report $\lambda_2(\hat P)$ to aid the design of future studies.  
%
%
%
% $\lambda_2$ 

\subsection{Simulations to compare the bootstrap procedures}

This section investigates the coverage properties of the confidence intervals generated from \atbns, \utbns, \achbns, and \ssbns.    
%\ssb fits and resamples from 
The successive-sampling (SS) model was first described for RDS in \cite{gile2011improved}. The \ssb fits and resamples from the SS model and was introduced in the \texttt{R} package \texttt{RDS} \citep{rdspackage}.  
%The key advantage of the SS model is that it samples without-replacement. 
The SS model is not Markovian and so it cannot be described as a $(\T, P)$\textit{-walk}.  The \ssb requires an estimate of the population size; in the simulations below, the function was provided with the true value of the population size.

\subsubsection{Simulation settings} \label{sec:simsettings}

In total, the figures below change three aspects of the simulations settings:  (i) the sample size of the RDS; (ii)  the strength of the bottleneck in $G$, i.e. $\lambda_2$;  and (iii) the strength of the relationship between the outcome $y$ and the bottleneck in $G$, i.e. $\rho_\pi^2(y,f_2)$ which is defined in Equation \eqref{def:rho}.  While the asymptotic properties of $Var_{RDS}$ only depend on whether $\rho_\pi^2(y,f_2)$ is zero or nonzero, the magnitude of $\rho_\pi^2(y,f_2)$ is highly relevant in a finite sample. 

To collect the desired sample size, each tree is initially sampled as a Galton-Watson tree with offspring distribution $1 + W$, where $W \sim Binomial(2, 1/2)$.  Then, the RDS sample is constructed without-replacement, using the procedure described in Section \ref{sec:replacementSim}.  So, $m = 2$ and the critical threshold is when the second eigenvalue is equal to $1/\sqrt{2}  \approx .71$.

To vary the value of $\lambda_2$, each network $G$ is simulated from a two block Stochastic Blockmodel  \citep{holland1983stochastic} with $70\%$ of the nodes in block $0$ with $z(i) = 0$ and $30\%$ of the nodes in block $1$ with $z(i) = 1$.  The size of the networks is set to $N=50,000$ and the probability of a connection between two nodes in different blocks is $r = 15/N$.  Then, $\lambda_2$ varies between $.5$ and $.9$ by varying the probability of a connection between two nodes in the same block.

To control $\rho_\pi^2(y,f_2)$, the simulations examine two types of node features $y$, \texttt{aligned} and \texttt{correlated}.  In the simulations for \texttt{aligned}, $y(i) = z(i)$ for all $i \in G$.  In the \texttt{correlated} simulation, $45\%$ of the nodes in block 0 have $y(i)=1$ and $10\%$ of the nodes in block 1 have $y(i)=1$; the rest of the nodes have $y(i) = 0$.   
%The names \texttt{aligned} and \texttt{correlated} come from the value of $\rho_\pi^2(y,f_2)$ which is defined in Equation \eqref{def:rho}.  
%In this simulation, $\rho_\pi^2(y,f_2)$ is random because $f_2$ depends on the randomness in the graph.  
In the \texttt{aligned} simulation, $\rho_\pi^2(y, f_2)$ is close to one.  In the \texttt{correlated} simulation, $\rho_\pi^2(y, f_2)$ is around $.15$.  See Figure \ref{fig:squaredCorrelation} in the appendix for more details.  

%$\langle y, f_2 \rangle_\pi^2$ in Equation \eqref{eq:var}.
%When $y(i)=z(i)$, the node feature aligns with the main bottleneck in the network.  In the second case, the node feature and bottleneck are correlated, but not perfectly.

The first step of the simulation is to generate the referral tree $\T$ from the Galton-Watson distribution with $n = 1000$ nodes. Then, the two types of node features $y$ are generated (\texttt{aligned} and \texttt{correlated}).  The $\T$ and $y$'s are fixed across all simulations.  Then, the following six steps create one replicate of the experiment:
\begin{enumerate}
\item Simulate the underlying network $G$ from a Stochastic Blockmodel.  To parameterize the Stochastic Blockmodel, set $r = 15/N$.   Then, with the desired $\lambda_2$, specify the edge probability $p$  via Equation \eqref{eq:sbmlambda2}.
\item Simulate a respondent-driven sample of 1000 nodes, without-replacement, using the more realistic model described in Section \ref{sec:replacementSim}, 
\item To examine sample sizes $n \ne 1000$, retain only the first $n$ samples.
\item Draw 500 samples from each of the resampling distributions (\atbns, \utbns, \achbns, and \ssbns).
\item Compute $\hat \mu_{VH}^*$ on each of the bootstrap samples.  
\item For each resampling distribution, use the 500 values of $\hat \mu_{VH}^*$ to compute the percentile confidence interval with the $5$th to the $95$th percentile of the bootstrap distribution for $\hat \mu_{VH}^*$.
\end{enumerate}
To examine the frequentist properties of these confidence intervals, the above five steps are repeated 501 times; 501 to avoid confusion with the number of bootstrap samples in step 3.

\subsection{Simulation results; the coverage probabilities of the confidence intervals}
Figures \ref{fig:2aligned} and \ref{fig:2correlated} display the estimated coverage probabilities as a function of the bottleneck strength $\lambda_2$.  The three panels correspond to different sample sizes.  Figure \ref{fig:2aligned} displays the results for \texttt{aligned} $y$.  Figure \ref{fig:2correlated}  gives the results for \texttt{correlated} $y$.  While all of the confidence intervals are nominally $90\%$, the figures show that the actual coverage probabilities can deviate substantially from $90\%$.  

%The \utb creates conservative intervals for smaller values  of $\lambda_2$ and anti-conservative intervals for larger values of $\lambda_2$. 
 Across simulation settings, \achb produces nominally $90\%$ confidence intervals which have coverage  probabilities ranging from $90\%$ to $10\%$.  These coverage probabilities are small in situations where the bottleneck is strong.  This demonstrates the sensitivities of \achb discussed in Section \ref{sec:reinterpreting} and in \cite{goel2010assessing}.  \ssb has coverage probabilities close to $90\%$ when $y$ is \texttt{aligned} and $\lambda_2$ is not too large. However, when $y$ is \texttt{correlated}, the coverage probabilities for \ssb quickly diminish for moderate to large values of $\lambda_2$.  The \utb confidence intervals are conservative for small values of $\lambda_2$ and anti-conservative for larger values of $\lambda_2$.   The intervals produced by \atb are conservative across the simulation settings.  

Note that in these simulations, the intervals from \atb appropriately  cover $\mu_{true}$ (i.e. not $\E(\hat \mu_{VH}|X_0)$).  Similar to Remark \ref{remark:seedbias}, these intervals account for the uncertainty due to seed selection (sometimes called seed-bias).

The results for the ``studentized" confidence intervals were studied, but are not displayed.  The ``studentized" confidence intervals are constructed as $\hat \mu_{VH} \pm 1.65  \hat \sigma$, where $\hat \sigma$ is the standard error of $\hat \mu_{VH}^*$ in the 500 bootstrap samples. In the simulations, the studentized intervals from the \atb often fail to be contained in $[0,1]$, despite the fact that $y_i \in \{0,1\}$ for all nodes $i$.  Perhaps one reason for these strange results is that the accuracy of the studentized intervals depends on $\hat \mu_{VH}$ being asymptotically normal, while results in \cite{li2015central} suggest that it is not. In the limited simulations that were performed, the percentile interval was often (i) narrower and (ii) more likely to cover $\mu_{true}$ than the studentized interval. The percentile interval can simultaneously improve both of these metrics because it is not necessarily symmetric around the point estimate.

%
%\begin{figure}[htbp] %  figure placement: here, top, bottom, or page
%   \centering
%Coverage probabilities for nominally $90\%$  intervals when outcome is \texttt{aligned} with the network.
%   
%   \includegraphics[width=6in]{sim2-perfectOverlap-2-16.pdf} 
%   \caption{In these simulations, $y$ is perfectly \texttt{aligned} with $z$, the referral bottleneck in the graph.  Across different sample sizes and varying strengths of referral bottlenecks, \atb creates confidence intervals with conservative coverage probabilities.}   \label{fig:2aligned}
%\end{figure}
%
%\begin{figure}[htbp] %  figure placement: here, top, bottom, or page
%   \centering
%   
%Coverage probabilities for nominally $90\%$ intervals when outcome is \texttt{correlated} with the network.
%      
%      
%   \includegraphics[width=6in]{sim2-notperfectOverlap-2-16.pdf} 
%   \caption{In these simulations, $y$ is \texttt{correlated} with $z$, the referral bottleneck in the graph.  Across different sample sizes and varying strengths of referral bottlenecks, \atb creates confidence intervals with conservative coverage probabilities.}
%   \label{fig:2correlated}
%\end{figure}

Section \ref{sec:widths} in the appendix presents another simulation which studies the widths of the confidence intervals.  Using a network with $\lambda_2 \approx .82$ (i.e. beyond the critical threshold), it studies how the width of the confidence interval decays as the sample size increases.  See Figure \ref{fig:simWidth} in the appendix for more details.

%\input{text/oldEmpiricalIllustration.tex}
%
%
%\input{text/bootstrapSection.tex}
%\input{text/estimationSection.tex}
%

%\newpage
\section{Discussion} \label{sec:discuss}

A common concern in the RDS literature has been the design effect of network sampling techniques \citep{salg, goel2010assessing, szwarcwald2011analysis, johnston2013empirical, verdery2013network}.  Theorems \ref{thm:var} and \ref{thm:critical} use the Markov model to give a rigorous account of the variance and design effect of RDS.  In particular, if $m>\lambda_2^{-2}$, then the design effect can grow with the sample size; this is equivalent to saying the the variance of the estimator decays slower than $O(n^{-1})$.  For two reasons, if the design effect is growing, then it should not be used for sample size or power calculations.   First, there might not be a central limit theorem to justify this approach \citep{li2015central}.  Second, if $DE$ changes with $n$, then many of the standard formulas are not well defined (or they are incorrect).  Instead of using $DE$ to summarize the quality of the sample, a more reasonable summary would be the ``half-life of the standard error.''  That is, given an RDS with sample size $n$, how much larger should $\tilde n$ be such to decrease the standard error by $50\%$.  For example, estimators which are $\sqrt{n}$-consistent (i.e. constant $DE$) have a half-life of $4$.  Past the critical threshold in RDS, the standard error decays like $n^{\gamma}$, where $\gamma = \log_m \lambda_2$ and  $-1/2 <\gamma<0$.  This means that the half-life of the standard error is $(1/2)^{(1/\gamma)} > 4$.

Section \ref{sec:replacement} examines how well the \tp (which samples with-replacement) approximates a more accurate simulation model (which samples without-replacement).  Proposition \ref{prop:lowerRn} and Theorem \ref{thm:upperRn} give matching lower and upper bounds on  the expected number of repeated pairs in a \tpns. 
So long as $n = o(\sqrt{N})$, and some  further technical conditions, these bounds show that $\lambda_2$ and the critical threshold do not affect the rate of $\E(R_n)$. As such, the critical threshold does not create additional repeated pairs.  Subsection \ref{sec:replacementSim} presents a simulation comparing the \tp to a network sample taken without-replacement.  Under the simulation settings, both the with-replacement and without-replacement samples displayed a similar critical threshold.  

Section \ref{sec:bootstrap} introduces \atbns, a new resampling procedure for computing confidence intervals for $\hat \mu_{VH}$.  In a wide range of simulation settings,  the intervals from \atb produced  intervals with conservative coverage probabilities (i.e. the nominally $90\%$ intervals had actual coverage that exceeded $90\%$).  In contrast, there were simulation settings under which \achbns, \ssbns, and \utb produced intervals with coverage probabilities that fall short of their nominal values.   
A key advantage of the \utb and \ssb is that they do not require $z$.  
In contrast, a key practical limitation of the \achb is that it requires a choice of $z$; 
that is, we must identify the referral bottleneck.  
%If choosing amongst a large set of possible variables to use for $z$, we want to choose the one that provides the largest value of $\lambda_2(\hat P)$ while also ensuring that the class sizes (i.e. the $n(u)$ values) are roughly balanced.  
More research is needed to (1) make \utb and \ssb less sensitive to $\lambda_2$ and (2) guide the choice of $z$ for \atbns.

\begin{figure}[htbp] %  figure placement: here, top, bottom, or page
   \centering
Coverage probabilities for nominally $90\%$  intervals when outcome is \texttt{aligned} with the network.
   
   \includegraphics[width=6in]{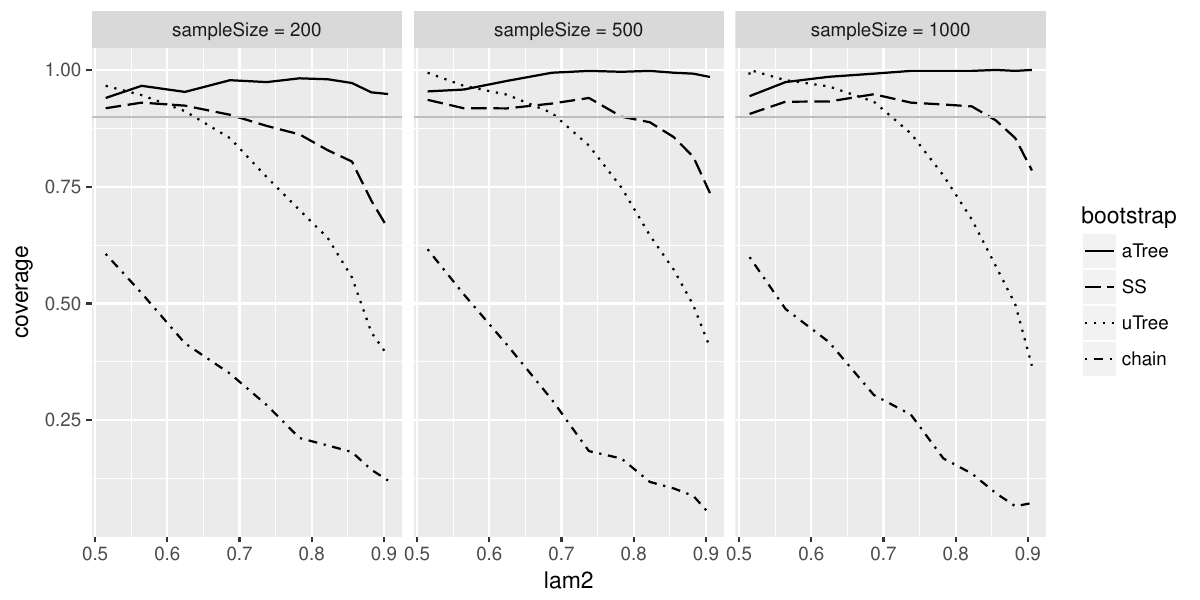} 
   \caption{In these simulations, $y$ is perfectly \texttt{aligned} with $z$, the referral bottleneck in the graph.  Across different sample sizes and varying strengths of referral bottlenecks, \atb creates confidence intervals with conservative coverage probabilities.}   \label{fig:2aligned}
\end{figure}

\begin{figure}[htbp] %  figure placement: here, top, bottom, or page
   \centering
   
Coverage probabilities for nominally $90\%$ intervals when outcome is \texttt{correlated} with the network.

   \includegraphics[width=6in]{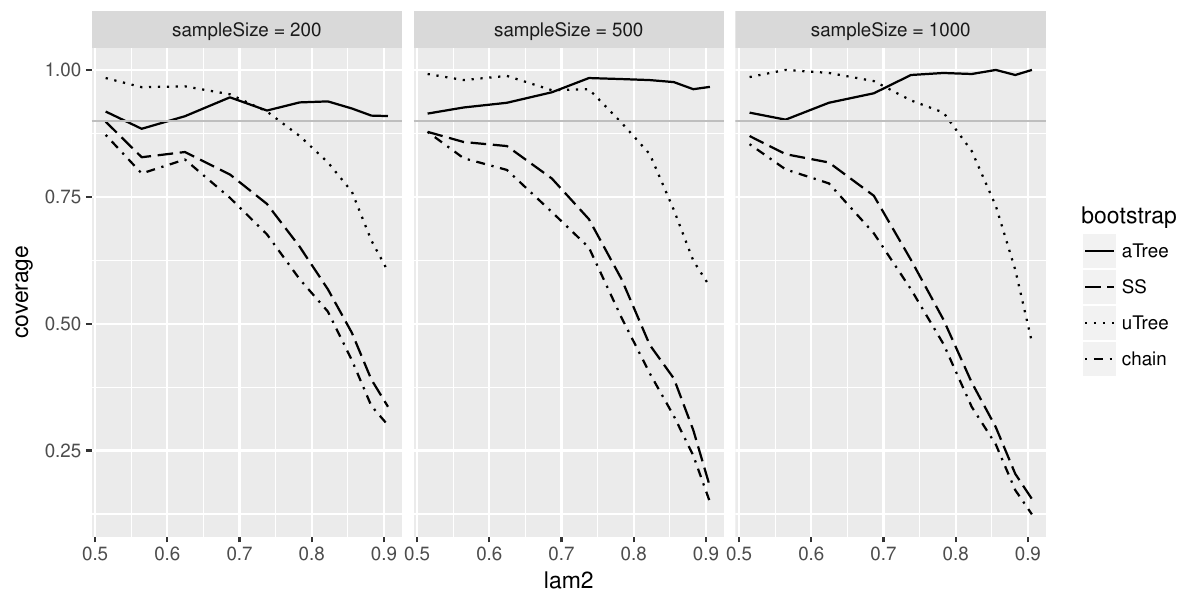} 
   \caption{In these simulations, $y$ is \texttt{correlated} with $z$, the referral bottleneck in the graph.  Across different sample sizes and varying strengths of referral bottlenecks, \atb creates confidence intervals with conservative coverage probabilities.}
   \label{fig:2correlated}
\end{figure}

%\newpage

\newpage
\appendix

\section{Proof of Theorem 2.1} \label{varProof}

The proof requires some notation and the following lemma.  Throughout, let $\{X_\sigma : \sigma \in \T\}$ be a \tpns.  Let $\{X(i): i \in 0, 1, \dots\}$ be a Markov chain with the same transition matrix $P$ that is initialized from $\pi$.   
Define $d(\sigma, \tau)$ as the graph distance between  nodes $\sigma$ and $\tau$ in $\T$.
\begin{lemma} \label{lem:rev}
If the transition matrix $P$ is reversible, then for any two nodes $\sigma$ and $\tau$ in the referral tree,
\[\pr(X_\sigma = u, X_\tau = v) = \pr\left(X(0) =u, X(d(\sigma,\tau)) = v\right).\]
\end{lemma}
\begin{proof}
Let $p = \sigma \wedge \tau$ be the most recent common ancestor of $\sigma$ and $\tau$.  By the reversibility of the process,
\begin{eqnarray*}
\pr(X_\sigma = u, X_\tau = v) &=& \sum_\ell \pr(X_\sigma= u, X_p = \ell, X_\tau = v)\\
&=& \sum_\ell \pi_\ell \pr(X_\sigma = u| X_p = \ell) \pr(X_\tau = v| X_p = \ell)\\
&=& \sum_\ell \pi_u \pr(X_p = \ell|X_\sigma = u) \pr(X_\tau = v| X_p = \ell)\\
&=& \sum_\ell \pi_u \pr( X(d(\sigma,p)) = \ell| X(0) = u) \pr( X(d(p,\tau) + d(\sigma,p)) = v| X(d(\sigma,p)) = \ell)\\
&=& \pr(X(0) =u,  X(d(\sigma, \tau)) = v).
\end{eqnarray*}
\end{proof}

Also, we require a fuller version of Lemma \ref{lem:tmp}, which comes from \cite{levin2009markov}.  

\begin{lemma} (Lemma 12.2 in \cite{levin2009markov}) \label{lem:spec}
Let $P$ be a reversible Markov transition matrix on the nodes in $G$ with respect to the stationary distribution $\pi$.  The eigenvectors of $P$, denoted as $f_1, \dots, f_{\gn}$, are real valued functions of the nodes $i \in G$ and orthonormal with respect to the inner product 
\begin{equation} \label{def:innerAppendix}
\langle f_a, f_b \rangle_\pi = \sum_{i \in G} f_a(i) f_b(i) \pi_i.
\end{equation}
If $\lambda$ is an eigenvalue of $P$, then $|\lambda|\le 1$.  The eigenfunction $f_1$ corresponding to the eigenvalue $1$ can be taken to be the constant vector $\textbf{1}$, in which case the probability of a transition from $i\in G$ to $j \in G$ in $t$ steps can be written as
\begin{equation}\label{eq:tsteps}
\pr(X(t) = j|X(0) = i) = P_{ij}^t =  \pi_j + \pi_j \sum_{\ell =2}^{\gn} \lambda_\ell^t f_\ell(i) f_\ell(j).
\end{equation}
\end{lemma}

The following is a proof of Theorem \ref{thm:var}.
\begin{proof} 
\begin{eqnarray*}
Var_{RDS}(\hat \mu) &=& \frac{1}{n^2} Var_{RDS}(\sum_{\tau \in \T}^N y(X_\tau)) \\
&=& \frac{1}{n^2}  \sum_{\sigma, \tau \in \T} Cov_{RDS}(y(X_\sigma), y(X_{\tau})).
\end{eqnarray*}

For ease of notation, let $t = d(\sigma, \tau)$.  From Lemma \ref{lem:rev} (and suppressing the $_{RDS}$ subscript),

\begin{eqnarray*}
Cov(y(X_\sigma), y(X_{\tau})) 
%&=& \E\left( (y(X_i) - \E_\pi(y(X_i))) \ (y(X_{j}) - \E_\pi(y(X_{j})))\right)\\
%&=& \E_\pi\left( y(X_i) \ y(X_{j}) \right) - \E_\pi(y(X_i) \E_\pi(y(X_{j}))) \\
%&=&
 = \E \left( y(X(0)) \ y( X(t))\right) -  \left( \E y( X(0))\right)^2.
\end{eqnarray*}

Using the spectral decomposition of $P$ (see Lemma \ref{lem:spec}), with the fact that $f_1$ is a constant vector and $\lambda_1 =1$ \citep{levin2009markov},
%The second term is equivalent to the terms from $Var(\hat \mu_{RS})$. For the first term, use $\mu = 0$ and facts about Markov chains from Chapter 12 of \cite{levin2009markov}.

\begin{eqnarray*}
\E\left( y(X(0)) y(X(t))\right) &=& \sum_{u,v \in G} y(u) y(v) \pr\left(X(0) = u, X(t) = v\right)\\
&=& \sum_{u,v \in G} y(u) y(v) \pi_u P^t(u,v)\\
&=& \sum_{u,v \in G} y(u) y(v) \pi_u \pi_v  \sum_{\ell=1}^{\gn} \lambda_\ell^t f_\ell(u) f_\ell(v) \\
&=& \sum_{u,v \in G} y(u) y(v) \pi_u \pi_v \left(1 +  \sum_{\ell=2}^{\gn} \lambda_\ell^t f_\ell(u) f_\ell(v)  \right)\\
&=& \left( \sum_{u\in G} y(u)  \pi_u\right)^2 +  \sum_{\ell=2}^{\gn} \lambda_\ell^t \left(\sum_{u \in G} y(u)  \pi_u   f_\ell(u) \right)^2\\
&=& \left(\E y( X(0))\right)^2 +  \sum_{\ell=2}^{\gn} \lambda_\ell^t  \langle y, f_\ell\rangle_\pi^2.
\end{eqnarray*}
Terms cancel.  So,
\begin{equation}
Cov(y(X_\sigma), y(X_{\tau}))= \sum_{\ell=2}^{\gn} \lambda_\ell^{d(\sigma, \tau)}  \langle y, f_\ell \rangle_\pi^2.
\end{equation}
Then,
\begin{eqnarray*}
Var_{RDS}(\hat \mu)  &=& n^{-2}  \sum_{\sigma, \tau \in \T} Cov(y(X_\sigma), y(X_\tau)) \\
 & = & n^{-2} \sum_{\sigma, \tau \in \T} \sum_{\ell=2}^{\gn} \lambda_\ell^{d(\sigma, \tau)}  \langle y, f_\ell \rangle_\pi^2 \\
 & = & n^{-2} \sum_{\ell=2}^{\gn} \langle y, f_\ell \rangle_\pi^2 \sum_{\sigma, \tau \in \T}  \lambda_\ell^{d(\sigma, \tau)}   \\
 & = & \sum_{\ell=2}^{\gn} \langle y, f_\ell \rangle_\pi^2 \G(\lambda_\ell) . 
\end{eqnarray*}

\end{proof}

%\section{Proof of Corollary \ref{cor:DEbounds}} \label{app:thresh}

%Corollary \ref{cor:varbounds} requires Lemmas \ref{lem:levin} and \ref{lem:thresh}.  First, a proof of Lemma \ref{lem:levin}.
%\begin{lemma} \label{lem:levin} For $\sigma^2 = Var_\pi(Y_0)$,
%\[\sigma^2 = \sum_{j=2}^{\gn}  \langle y, f_j\rangle_\pi^2 . \]
%\end{lemma}
%\begin{proof}
%This proof is given on page 342 of \cite{levin2009markov} and is repeated here for completeness. 
%\[\sum_{j=2}^{\gn}  \langle y, f_j\rangle_\pi^2  = \sum_{j=1}^{\gn}  \langle y, f_j\rangle_\pi^2 - (\E_\pi Y_0)^2 =  \E_\pi(Y_0^2) - (\E_\pi Y_0 )^2 = Var_\pi (Y_0).\]
%\end{proof}

%\begin{lemma} \label{lem:gpos}
%For any tree and any \tp with a reversible $P$,
%\[\G(\lambda_k) \ge 0 \mbox{ for all } k = 1, \dots, N.\]
%\end{lemma}
%The following is a proof of Lemma \ref{lem:gpos}.
%\begin{proof}
%Define $y = f_k$.  By Theorem \ref{thm:var}, 
%\[Var_{RDS}(\hat \mu)  =  \sum_{\ell=2}^{\gn} \langle y, f_\ell \rangle_\pi^2 \G(\lambda_\ell) = \G(\lambda_k).\]
%Moreover $Var_{RDS}(\hat \mu) \ge 0$. If $k=1$, then 
%\end{proof}

%\begin{eqnarray*}
%Var_{RDS}(\hat \mu)  &=&  \sum_{\ell=2}^{\gn} \langle y, f_\ell \rangle_\pi^2 \G(\lambda_\ell) \\
%&\le&  \G(\lambda_2) \sum_{\ell=2}^{\gn} \langle y, f_\ell \rangle_\pi^2  \\
%&=&  \G(\lambda_2) \sigma^2  
% \end{eqnarray*}

\section{Proof of Theorem 3.1} \label{sec:ProofCritical}

Sections \ref{sec:PrelimCritical}, \ref{sec:lowerBound}, and \ref{sec:upperBound} all give conditions on the topology of the tree $\T$.  Then, Section \ref{sec:GWP} shows that Galton-Watson trees almost surely have these topological properties. 

\subsection{Preliminaries to the proof} \label{sec:PrelimCritical}

The next corollary shows that if the outcome of interest $y$ correlates with the largest bottleneck in the network $f_2$, then the design effect is asymptotically proportional to $n \G_n(\lambda_2)$.  As such, the \tp has a bounded design effect if and only if $\G_n(\lambda_2) = O(n^{-1})$.   Recall that the eigenvalues are defined in descending absolute value, $|\lambda_1| \ge |\lambda_2| \ge \dots \ge |\lambda_N|$.
 
\begin{corollary} \label{cor:DEbounds}
Under the conditions of Theorem \ref{thm:var}, if $\lambda_2 > 0$, then
\[\rho_\pi^2(y, f_2) \ n  \G_n(\lambda_2)  \le DE(\hat \mu) \le n  \G_n(\lambda_2),\]
where $\rho_\pi^2(y, f_2) $ is the population correlation between $y$ and the second eigenvector of $P$,
\begin{equation}\label{def:rho}
\rho_\pi(y, f_2) = \sigma^{-1}\langle y,  f_2 \rangle_\pi,
\end{equation}
for $\sigma^2 = Var_{RDS}Y_0$ and $\langle \cdot, \cdot \rangle_\pi$ as defined in Equation \eqref{def:innerAppendix}.
\end{corollary}

The proof of Corollary \ref{cor:DEbounds}, uses two lemmas.
\begin{lemma} \label{lem:levin} For $\sigma^2 = Var_{RDS}(Y_0)$,
\[\sigma^2 = \sum_{j=2}^{\gn}  \langle y, f_j\rangle_\pi^2 . \]
\end{lemma}
This proof is given on page 342 of \cite{levin2009markov} and is repeated here for completeness. 
\begin{proof}
\[\sum_{j=2}^{\gn}  \langle y, f_j\rangle_\pi^2  = \sum_{j=1}^{\gn}  \langle y, f_j\rangle_\pi^2 - (\E_\pi Y_0)^2 =  \E_\pi(Y_0^2) - (\E_\pi Y_0 )^2 = Var_\pi (Y_0)= Var_{RDS} (Y_0).\]
\end{proof}
\begin{lemma} \label{lem:gpos}
For any \tp that satisfies the conditions of Theorem \ref{thm:var},
\[\G(\lambda_k) \ge 0 \mbox{ for all } k = 1, \dots, N.\]
\end{lemma}
\begin{proof}
Define a pretend feature $y = f_k$, then by Theorem \ref{thm:var}, $\G(\lambda_k) = Var_{RDS}(\hat \mu)  \ge 0.$ 
%For $k=1$, $\G(1) = 1$.
\end{proof}

Now, a proof of Corollary \ref{cor:DEbounds}.
\begin{proof}
By the definition of the ordering, $|\lambda_1| \ge |\lambda_2| \ge \dots \ge |\lambda_N|$ and the assumption $\lambda_2 > 0$, it follows that $\lambda_2 \ge |\lambda_{\ell}|$ for $\ell > 2$.  This implies $\lambda_2^d > \lambda_\ell^d$ for any $d$.  It then follows that $\G(\lambda_2)  \ge\G(\lambda_\ell)$. So,
\[Var_{RDS}(\hat \mu)  =  \sum_{\ell=2}^{\gn} \langle y, f_\ell \rangle_\pi^2 \G(\lambda_\ell) \le  \G(\lambda_2) \sum_{\ell=2}^{\gn} \langle y, f_\ell \rangle_\pi^2  =  \G(\lambda_2) \sigma^2.  \]
Because $\G(\lambda_\ell) \ge 0$ for all $\ell$,
\[Var_{RDS}(\hat \mu)  =  \sum_{\ell=2}^{\gn} \langle y, f_\ell \rangle_\pi^2 \G(\lambda_\ell)  \ge  \langle y, f_2 \rangle_\pi^2 \G(\lambda_2).\]
To convert to $DE$, divide by $Var_\pi(\hat \mu) = Var_{RDS}(Y_0) /n = \sigma^2 / n$.  
\end{proof}

\subsection{Lower bounds for $\G(\lambda_2)$} \label{sec:lowerBound}

\begin{fact} \label{fact:Gbounds}
Select two nodes $I, J$ from $\T$ uniformly at random.   Define the random variable $D = d(I,J)$ to be the graph distance in $\T$ between $I$ and $J$. Define $\|J\| = d(0,J)$ to be the distance from $J$ to the root.  For $\lambda \in [0,1)$,
 \[\G(\lambda) \ge \lambda^{\E D}  \ge \max(\lambda^{d(\T)}, \lambda^{2 \E \|J\| }) \ge \min(\lambda^{d(\T)}, \lambda^{2 \E \|J\| }) \ge \lambda^{2 h(\T)},\]
 where $\E \|J\|$ is the average distance from the seed node, $d(\T)$ is the diameter of the $\T$, and $h(\T)$ is the height of the tree.
\end{fact}

%The following is a proof of Fact \ref{fact:Gbounds}.
\begin{proof}
The first inequality follows directly from Jensen's inequality.  The next inequalities use 
\[\E D \le \E (\|I\| +\|J\|) = 2 \E \|J\| \le 2 h(\T).\]
Also, notice that $\E D \le d(\T) \le 2 h(\T)$.
The result follows from the restriction that $|\lambda|<1$.
\end{proof}

%Empirical RDS papers often report $h(\T)$ and $\E \|J\|$ because they help to diagnose whether the process has reached the stationary distribution.  Fact \ref{fact:Gbounds} shows that they are also useful for lower bounding the variance.  However, the average pairwise distances $\E D$ or the entire function $\G(z)$ would be more informative. 

%Let $\beta = \lambda_2^{-2}$.  Fact \ref{fact:Gbounds} implies that if 
%\[d(\T_n) \le 2 \alpha \log_\beta n, \ \mbox{ or if, } \ h(\T_n) \le \alpha \log_\beta n,\]
%then
%\[\G(\lambda_2) \ge n^{-\alpha}.\]

%\end{equation}
%for $\beta = \lambda_2^{-2}$, then
%\[Var_{RDS}(\hat \mu)  \ge c n^{-(1-\epsilon)} + O(\G(\lambda_3) )\]
%for some constant $c$.
%\end{corollary}
%Note that $2 h(\T_n) \ge d(\T_n)$.  So, a sufficient condition for inequality \eqref{eq:dbound} is 
%= \log_m \beta \log_\beta n
%\[.\]

Define $\beta = 1/\lambda_2^2$.  
%Fact \ref{fact:Gbounds} shows that $\hat \mu$ has a growing design effect when $\T$ is an $m$-tree and $m >\beta$.  

$m$-trees provide good intuition for trees that grow at rate $m$.  If $\T$ is an $m$-tree, 
%with $m > \beta$, then 
then $h(\T) \le \log_m n $.  So,
\begin{equation} \label{eq:Galpha}
\G(\lambda_2) \ge  \lambda_2^{2 h(\T)} \ge  \lambda_2^{2 \log_m n} = n^{- \log_m \beta}.
\end{equation}
%where the second inequality follows from the fact that when $\T$ is an $m$-tree, $h(\T) \le \log_m n $.  
%for $\alpha = \log_m \beta $.
%This implies that the design effect grows when $m > \beta$.  
The next fact shows that the lower bound in Equation \eqref{eq:Galpha} is not tight when $m<\beta$.  
\begin{fact} \label{fact:oneovern}
When $\lambda>0$, $\G(\lambda) \ge n^{-1}.$
\end{fact}
%The following is a proof of Fact \ref{fact:oneovern}.
\begin{proof}
As before, denote $D = d(I,J)$, then
\[\G(\lambda) = \sum_{k = 0}^d\lambda^k\pr(D=k) \ge \lambda^0 \pr(D=0) = n^{-1}.\]
\end{proof}

Taking the maximum of these two lower bounds shows that for $m$-trees, 
%is $max(n^{-1}, n^{- \log_m \beta})$.  When 
the lower bound is $n^{-1}$ when $m<\beta$ and $n^{- \log_m \beta}$ when  $m > \beta$.  

%Together Facts \ref{fact:Gbounds} and \ref{fact:oneovern} show that when $m <\beta$, the design effect does not converge to zero, and when $m > \beta$ the design effect grows with $n$.  Said another way, when $m < \beta$, the estimators do not converge faster than the standard $\sqrt{n}$ rate, and when $m > \beta$, the estimators converge slower than the standard $\sqrt{n}$ rate.  These are lower bounds.
%%Above this bound, $\G(\lambda_2)$ does not converge at the standard rate. Under this regime, VH would have an unbounded design effect.  
The next section gives a matching upper bound under an additional ``balanced" condition on $\T$.

%
%Taken together, the lower bounds suggest a critical threshold for the convergence of $\G(\lambda_2)$ on $m$-trees.  For slower referral rate (i.e. $m < \beta $),  $\G$ is lower bounded by the standard rate of convergence, $\G(\lambda_2) \ge n^{-1}$.  However, higher referral rates (i.e. $m > \beta $) produce a sequence of $\G_n(\lambda_2)$ that do not obtain the standard rate.  In fact, $\G(\lambda_2) > n^{-\alpha}$ where $\alpha = \log_m \beta <1$. This critical threshold 
%%So, when $\T$ is an $m$-tree and $m>\lambda_2^{-2}$, the function $\G(\lambda_2)$ does not decay at the standard $n^{-1}$.  This suggests that there is a critical threshold at $m>\lambda_2^{-2}$.  
%will be confirmed with upper bounds on $\G(\lambda_2)$ in the next section.  

\subsection{Upper bound for $\G(\lambda_2)$} \label{sec:upperBound}
%Under the perfect $m$-tree, the lower bounds show that $\G(\lambda_2)$ converges a slow rate when $m> \lambda_2^{-2}$.  

Upper bounding $\G$ requires a more global assumption about the ``balance" of $\T$.   Note that $\G(\lambda_2)$ is small when $d(I,J)$ is likely to be big (i.e. $\hat \mu$ has a smaller variance  when most distances are large).  

To see the necessity of an additional condition for an upper bound,  suppose that a tree grows at rate $m>1$ and in every generation $t-1$, there is a single node that produces all the nodes in generation $t$.  Because $m>1$ there is a non-vanishing probability that $I$ and $J$ come from the final generation $h$.  On this event, $I$ and $J$ have the same parent and $D = d(I,J) = 2$.   As $h$ grows, $\G$ will not decay.  An upper bound that decays with the lower bounds, it is necessary to prevent this type of tree.   Define $\|I\| = d(0,I)$ to be the distance from the root to node $I$.  For  $\tau, \sigma \in \T$, define $\tau \wedge \sigma \in \T$  to be the most recent common ancestor of $\sigma$ and $\tau$.
%\footnote{$  \tau \wedge \sigma = arg max \|\gamma\| $ such that $\gamma \in 
The formula
\[d(I,J) = \|I\| + \|J\|- 2\|I\wedge J\|\]
shows that most pairwise distances $d(I,J)$ are large when $\| I \|$  is large for most nodes and when $\|I\wedge J\|$ is small for most pairs.  In essence, the balanced condition (which is defined below) ensures that $\|I\wedge J\|$ is small. 

For $\sigma \in \T$, define $\A(\sigma)$ as the set of ancestors of $\sigma$, that is the nodes in $\T$ that fall along the shortest path between $\sigma$ and the root (for convenience, include $\sigma \in \A(\sigma)$). 
%For $\tau \in \T$, 
%recall $\|\tau\| = d(0, \tau)$ and 
Define the descendants of $\tau \in \T$ in the $n$th generation as
\[\D_n(\tau) = \{\sigma: d(0,\sigma)=n \ \mbox{ and } \ \tau \in \A(\sigma)\}.\]

Because $0$ is the root node, $\D_n(0)$ contains all nodes in generation $n$ and $|\D_n(0)|$ is the number of nodes in generation $n$. A tree $\T$ \textbf{grows at rate $m$} if there exist positive constants 
$\cb$ and $\ct$ such that for all $n$,
\[\cb  \le \frac{|\D_n(0)|}{m^n} \le \ct.\]
Notice that this implies $\T$ is an infinite tree.  The results below study $\T_h$, the induced subgraph of $\T$ that is formed by all nodes  $\tau$ with $\|\tau\| \le h$.

Suppose that $\T$ grows at rate $m$.  For $\tau \in \T$ with $\|\tau\|=k$, define 
\[c_\tau = \sup_n \frac{|\D_n(\tau)|}{m^{n-k}}.\]
Because $|\D_n(\tau)| \le |\D_n(0)|$ and the tree is assumed to grow at rate $m$, these constants are finite; $c_\tau \le \ct m^k <\infty$.  However, under a sequence of $\tau_n$, $c_{\tau_n}$ could be unbounded.  A tree satisfies the \textbf{balanced assumption} if there exists a constant $c$ such that for all $n$,
\[|\D_n(0)|^{-1} \sum_{\|\tau\|= n} c_\tau^2 \le c < \infty.\]
That is, the second moment of the $c_\tau$'s is uniformly bounded across all generations.
For example, $m$-trees grow at rate $m$ and satisfy the balanced assumption because $c_\tau = 1$ for all $\tau$.  An assumption similar to the balanced condition has appeared previously; Proposition 3.3 in \cite{lyons1990} implies that a balanced tree is ``quasi-spherical.''

\begin{theorem} \label{thm:upperBound}
Let $\T$ be an infinite tree that grows at rate $m$.  Define $\T_h$ as the node induced subgraph of $\T$ that contains all nodes $\tau \in \T$ satisfying $\|\tau\| \le h$.  Define $\G_h$ as in Definition \ref{def:probgen} with tree $\T_h$. 
If $\T$ satisfies the balanced assumption, then 
\begin{equation} \label{eq:upperBound} \arraycolsep=1.4pt\def\arraystretch{1.4}
\G_{h}(\lambda_2) \le 
\left\{\begin{array}{ll}
c  (\log n) n^{-1} & \mbox{ if } m < \beta \\
c (\log n)^2 n^{-1}& \mbox{ if } m = \beta \\
c (\log n) n^{-\alpha} & \mbox{ if } m > \beta,
\end{array}\right.
\end{equation}
where $\beta = \lambda_2^{-2}$, $\alpha = \log_m \lambda_2^{-2}$, and $c$ is a constant that could depend on $m$ and $\lambda_2$, but is independent of $n$. 
\end{theorem}
%\end{LARGE}

The growth rate assumption implies that $\T_h$ has $n = O(m^h)$ nodes.  So, Fact \ref{fact:Gbounds} yields matching lower bounds; the $\beta$ threshold is identical and the rates differ only by $\log n$ terms.

The key to the proof of Theorem \ref{thm:upperBound} comes from upper bounding $\pr(d(I,J) = k)$, where $I$ and $J$ are nodes selected uniformly at random from $\T$.  First, condition on $\|I\|$ and $\|J\|$.  Then, $d(I,J) = \|I\| + \|J\|- 2\|I\wedge J\|$ is determined by $\|I\wedge J\|$.
 In order to use the fact that $I$ and $J$ are independent, 
\begin{eqnarray*}
\pr(\|I\wedge J\| = \ell \ | \ \|I\| = a, \|J\| = b) & = &
\sum_{\tau: |\tau|= \ell} \pr( \tau = I\wedge J  \ | \ \|I\| = a, \|J\| = b)\\ &\le&  
\sum_{\tau: |\tau|= \ell} \pr( \tau \in \mathscr{A}(I) \ | \ \|I\| = a) \pr(\tau \in \mathscr{A}(J)|  \|J\| = b).
\end{eqnarray*}
These terms are related to $c_\tau^2$. So, the balance condition provides a bound. Finally, the growth rate assumption provides bounds for $\pr(\|I\| = a)$.

The proof of Theorem \ref{thm:upperBound} uses the following fact about a finite geometric series:
\begin{equation}\label{eq:geoseries} 
\arraycolsep=1.4pt\def\arraystretch{1.4}
\sum_{k=0}^{2h} x^k = \frac{1 - x^{2h+1}}{1-x} 
 \le \left\{\begin{array}{ll}
(1-x)^{-1} & \mbox{ if } x < 1 \\
x^{2h+1} (x-1)^{-1} & \mbox{ if } x > 1.
\end{array}\right.
\end{equation}

The following is a proof of Theorem \ref{thm:upperBound}.
\begin{proof}
An upper bound on $\pr(d(I,J) = k)$ provides an upper bound on $\G_h(\lambda_2)$.
\begin{eqnarray*}
\pr(d(I,J) = k) 
&=& \sum_{j=0}^k \pr(d(I\wedge J, I) = k-j \cap d(I\wedge J, J) = j)\\
&=& \sum_{j=0}^k \sum_{\ell = 0}^{\lfloor h - k/2\rfloor} \pr(d(I\wedge J, I) = k-j \cap d(I\wedge J, J) = j \ | \ \|I\| = k-j + \ell, \|J\| = j + \ell) \\
&& \hspace{1in} \times  \ \pr( \|I\| = k-j + \ell) \pr( \|J\| = j + \ell).
\end{eqnarray*}
First, bound the terms on $\|I\|$ and $\|J\|$ with the growth rate assumption,
% what are a and b??? is this old notation that was replaced with \cb and \ct???
%\[\pr( \|I\| = k-j + \ell) \pr( \|J\| = j + \ell) \le a^2 b^2 m^{k+2\ell - 2h}.\]
\[\pr( \|I\| = k-j + \ell) \pr( \|J\| = j + \ell) \le \ct^2 m^{k+2\ell - 2h}.\]
To bound the $I \wedge J$ term, define
\[c_{\tau, n} = \frac{\mathscr{D}_{\|\tau\| + n}(\tau)}{m^n}.\]
A key idea in what follows is that $\tau = I\wedge J \implies \tau \in \A(I) \cap \A(J)$.  Then, because $I$ and $J$ are independent, this probability breaks apart into two terms.
\begin{eqnarray*}
&& \pr(d(I\wedge J, I) = k-j \ \cap \ d(I\wedge J, J) = j \ | \ \|I\| = k-j + \ell, \|J\| = j + \ell) \\
&=& \pr(\|I\wedge J\| = \ell \ | \ \|I\| = k-j + \ell, \|J\| = j + \ell)\\ & = &
\sum_{\tau: |\tau|= \ell} \pr( \tau = I\wedge J  \ | \ \|I\| = k-j + \ell, \|J\| = j + \ell)\\ &\le&  
\sum_{\tau: |\tau|= \ell} \pr( \tau \in \mathscr{A}(I) \ | \ \|I\| = k-j + \ell) \pr(\tau \in \mathscr{A}(J)|  \|J\| = j + \ell)\\ &=&  
\sum_{\tau: |\tau|= \ell} \frac{\mathscr{D}_{k-j +\ell}(\tau)}{\mathscr{D}_{k-j+\ell}(0)} \frac{\mathscr{D}_{j + \ell}(\tau)}{\mathscr{D}_{j+\ell}(0)}\\ &\le&  
%a^2 \sum_{\tau: |\tau|= \ell} c_{\tau, k-j} c_{\tau,j} m^{-2\ell}. % what is a???? I think this is old notation for 1/\cb.  
1/\cb^2 \sum_{\tau: |\tau|= \ell} c_{\tau, k-j} c_{\tau,j} m^{-2\ell}.
\end{eqnarray*}
By the definition of $c_\tau$,
\[\sum_{j=0}^k c_{\tau, k-j} c_{\tau,j} \le k c_\tau^2.\]
So,
\begin{eqnarray*}
\pr(d(I,J) = k) 
&\le& c \sum_{j=0}^k \sum_{\ell = 0}^{\lfloor h - k/2\rfloor} \sum_{\tau: |\tau|= \ell} c_{\tau, k-j} c_{\tau,j} m^{-2\ell} m^{k+2\ell- 2h}\\
&\le& c m^{k- 2h} k
 \sum_{\ell = 0}^{\lfloor h - k/2\rfloor} m^\ell (\mathscr{D}_\ell(0))^{-1} \sum_{\tau: |\tau|= \ell} 
 c_\tau^2.
\end{eqnarray*}
By the balanced assumption, there is a constant $c < \infty$ such that for all $\ell$,
\[(\mathscr{D}_\ell(0))^{-1} \sum_{\tau: |\tau|= \ell}  c_\tau^2 < c.\]
So, use Equation \eqref{eq:geoseries} and let the constant depend on $m$,
\begin{eqnarray*}
\pr(d(I,J) = k) &\le& c m^{k- 2h} k \sum_{\ell = 0}^{\lfloor h - k/2\rfloor} m^\ell  \\
&\le& c k m^{k- 2h}   m^{ h - k/2 + 1}  \\
&=& c k m^{k/2- h}. 
\end{eqnarray*}
By the growth rate assumption, $m^{-h} \le c n^{-1}$ and $h \le  c \log n$.  So,
\begin{eqnarray*}
\G_h(z) &=& \sum_{k= 0}^{2h} z^k \pr(d(I,J) = k) \\
&\le&  c \sum_{k= 0}^{2h} z^k k m^{k/2- h} \\
&\le&c n^{-1} \log n \sum_{k= 0}^{2h}  (\sqrt{m} z)^k .
%&=& c n^{-1} \log n \frac{(\sqrt{m} z)^{2h+1} -1}{\sqrt{m} z -1} \ \ \mbox{ for $\sqrt{m} z \ne 0$.}
\end{eqnarray*}
When $mz^2 =1$, the sum contributes $2h \le c \log n$ and the rate is $n^{-1} (\log n)^2$.  Using the fact about geometric series in Equation \eqref{eq:geoseries}, for $mz^2 \ne1$,
\[\G_h(z) \le c n^{-1} \log n \frac{(\sqrt{m} z)^{2h+1}}{\sqrt{m} z -1}.\]
When $mz^2 <1$, the leading term gives the rate because the fraction converges to a constant.      However, when $mz^2 >1$, the fraction explodes.
\[n^{-1}\frac{(\sqrt{m} z)^{2h+1}}{\sqrt{m} z -1} \le c n^{-1} (m z^2)^h\le c z^{2h} = c z^{2\log_m n} = c n^{2\log_m z} .\]

\end{proof}

\subsection{Galton-Watson trees satisfy the conditions of the above results} \label{sec:GWP}

Fact \ref{fact:Gbounds} shows that if $h(\T) < \log_m n$, then 
\[\G(\lambda_2) \ge n^{2\log_m\lambda_2}.\]
The Kesten-Stigum theorem shows $h(\T) \approx \log_m n$ holds for Galton-Watson trees with $\E(\xi) = m$.

\begin{theorem}\label{thm:ks} \citep{kesten1966limit} 
Suppose $\T$ is a random Galton-Watson tree.  Let $\xi$ be a single draw from the offspring distribution; presume that $m = \E(\xi) >1$ and $\E(\xi \log \xi) < \infty$. Conditioned on the survival of the Galton-Watson process, $\T$ grows at rate $m$, a.s..
\end{theorem}

See \cite{lyons1995conceptual} for a conceptual proof of the Kesten-Stigum Theorem.

Thus, under the conditions of Theorem \ref{thm:ks}, $\T_h$ has $n = O(m^h)$ nodes.  
As such, the function $\G_h$ built from $\T_h$  will have the same lower bound as $m$-trees (see the discussion after after Fact \ref{fact:Gbounds}).  A matching upper bound on $\G_h$ requires a fourth moment assumption on $\xi$.

%%\subsection{Galton-Watson $\T$}
%%Under the Galton-Watson  (GW) model for $\T$, each participant refers an iid number of participants.  This iid draw comes from the ``offspring distribution".  Let $\xi$ be a draw from this distribution and denote $\E \xi =m$.  
%\begin{theorem}\label{thm:gwp}
%Suppose $\T$ is sampled from a Galton-Watson branching process that is stopped after $h$ generations. Under this model, the sample size $N_h$ is random.  Let $m>1$ be the mean of the offspring distribution.  Condition on the survival of the GW process.  For large enough $h$, 
%\[\G_h(\lambda_2) \ge c_W N_h^{\alpha},\]
%where $\alpha = \log_m \lambda_2^2$ and $c_W>0$ is a ``founders effect" random variable that is independent of $h$ and described in the proof. 
%\end{theorem}
%The following is a proof of Theorem \ref{thm:gwp}.
%\begin{proof}
%Using the standard GW notation, let $Z_n$ be the number of nodes in $\T$ at distance $n$ from the root.  
%\[W_n = \frac{Z_n}{m^n}\]
%is a martingale and $\lim_n W_n = W$ exists \textit{a.s.}  
%\[\frac{N_h}{m^h} = \frac{\sum_{n \le h} Z_n }{m^h} > \frac{Z_h}{m^h} + \frac{Z_{h-1}}{m^h} \rightarrow W(1+1/m)\]
%By the Kesten-Stigum Theorem, conditioned on the survival of the GW, $W>0$. So, there exists a $b$ such that for all $h > b$, 
%\[\frac{Z_h}{m^h} + \frac{Z_{h-1}}{m^h} > W. \]
%Multiplying by $m^h$ and taking $\log$'s,
%\[\log_m N_h  > \log_m W+ h.\]
%Using fact \ref{fact:Gbounds}, 
%\[\G_h(\lambda_2) \ge \lambda_2^{2 h} \ge \lambda_2^{2\log_m N_h - 2\log_m W} = N_h^{\log_m \lambda_2^2} W^{- \log_m \lambda_2^2}\]
%\end{proof}

%\begin{LARGE}
\begin{lemma}\label{thm:gwp}
Suppose $\T$ is a random Galton-Watson tree.  Let $\xi$ be a single draw for the offspring distribution; presume that $m = \E(\xi) >1$ and $\E(\xi^4) < \infty$.
Conditioned on the survival of the Galton-Watson process, $\T$ satisfies the conditions of Theorem \ref{thm:upperBound} (i.e. it grows at rate $m$ and it is balanced, a.s.).
\end{lemma}
The proof relies on the fact that $c_\tau$ is the supremum of the standard Galton-Watson martingale. 
%Then, it relies upon the $L^p$ maximal inequality for martingales.  The full proof is contained in the appendix, Section \ref{app:gwp}. 
%\end{LARGE}
%
%
%
%
%
%
%XXXXXXX
%
%The following is a proof of Lemma \ref{thm:gwp}.
\begin{proof}
%\textbf{GW grows at rate $m$:} Using the standard GW notation, let $Z_n$ be the number of nodes in $\T$ at distance $n$ from the root.  
%\[W_n = \frac{Z_n}{m^n}\]
%is a nonnegative martingale.  Using $Var(\xi)<\infty$, it is a classical application of the Martingale Convergence Theorem (e.g. Example 5.4.3 in \cite{durrett}) to show that $\lim_n W_n = W$ exists \textit{a.s.}.\footnote{This also holds under weaker conditions.}  Conditioned on the survival of the process, $W>0$ by the Kesten-Stigum Theorem (e.g. \cite{kesten1966limit} or \cite{lyons1995conceptual} for a modern reference). Then, standard arguments imply that $\T$ will grow at rate $m$, conditioned on the survival of the process.  

%\textbf{GW is balanced:} 
Because trees that go extinct are balanced, it is not necessary to condition on survival.  The proof below shows that if $\T$ is generated from the Galton-Watson with a finite fourth moment, then it is balanced a.s..

Each node $\tau \in \T$ generates an identically distributed Galton-Watson tree below it.  Denote 
\[Z_n^\tau = |\D_{\|\tau\| + n}(\tau)|, \ \ \ W_n^\tau = \frac{Z_n^\tau}{m^n}, \ \ \ W_+^\tau = \sup_n \ W_n^\tau, \ \mbox{ and } \  W^\tau = \lim W_n^\tau.\]
Across all values of $\tau$, $W_+^\tau$ are identically distributed.  Moreover, within a single generation of the tree (i.e. $\tau \in \D_h(0)$), $W_+^\tau$ are independent.  The same holds for $Z_n^\tau, W_n^\tau$ and $W^\tau$.  So, dropping the superscript $\tau$ will correspond to a generic iid draw from the same distribution.

The values $W_+^\tau$ correspond to the $c_\tau$'s in the balanced assumption.  We wish to bound
\[|\D_h(0)|^{-1} \sum_{\|\tau\| = h} (W_+^\tau)^2 < C,\]
where $C$ is a random variable that does not depend on $h$.  

\begin{lemma} \label{lem:limW}
Under the conditions of the theorem, $\E W_+^4 <\infty$.  
\end{lemma}
A proof of this lemma is given following the proof of the theorem.

Using Borel-Cantelli, the argument below will show that for $\mu = \E W_+^2$ and $\epsilon>0$, 
\begin{equation}
\label{eq:bc}
\pr\left(\left\{|\D_h(0)\|^{-1} \sum_{\|\tau\| = h} (W_+^\tau)^2 > \mu + \epsilon\right\}\ i.o. \mbox{ in $h$ }\right) = 0.
\end{equation}
As such, a.s. there exists a variable $C(\omega)$ that satisfies the balanced condition. 

Denote $Z_k = Z_k^0$.  Let $\F_{h}$ denote the filtration $\sigma(Z_0, Z_1, \dots, Z_h)$. 
By Chebyshev's inequality,
\begin{eqnarray}
&&\nonumber \sum_{h=1}^\infty\pr\left( \sum_{\|\tau\| = h} (W_+^\tau)^2 > Z_{h-1}(\mu + \epsilon)\right) \\
\nonumber  
&=&\sum_{h=1}^\infty \E \pr \left( \sum_{\|\tau\| = h} (W_+^\tau)^2 > Z_{h-1}(\mu + \epsilon) | \F_{h-1} \right) \\
\nonumber  
&\le &\sum_{h=1}^\infty \E \frac{ \textbf{1}\{Z_{h-1} \ne 0\} \sum_{\|\tau\| = h} \E \left( \left((W_+^\tau)^2 -\mu\right)^2 | \F_{h-1} \right)}{ (Z_{h-1})^2 \epsilon^2} \\
&= &\sum_{h=1}^\infty \E \frac{ \textbf{1}\{Z_{h-1} \ne 0\} \E \left( \left(W_+^2 -\mu\right)^2 | \F_{h-1} \right)}{ Z_{h-1} \epsilon^2}. \label{eq:lastone}
%XXX &\le& c \sum_{h=1}^\infty Var\left( (Z_{h-1}^0)^{-1} \sum_{\|\tau\| = h} (W_+^\tau)^2\right) \\
%&\le& c \sum_{h=1}^\infty \E\left( (Z_{h-1}^0)^{-2}\sum_{\|\tau\| = h} \E \left(((W_+^\tau)^2 - \mu )^2|\F_{h-1}\right)\right). \label{eq:lastone}
\end{eqnarray}
%
%The following calculations divide by $Z_{h-1}^0$.  On the event that $Z_{h-1}=0$ for some $h$, then  the balanced assumption is satisfied. As such, the following calculations could include an indicator for the event $Z_{h-1}^0>0$, but this is suppressed.  By Chebyshev's inequality,
%\begin{eqnarray}
%\nonumber \sum_{h=1}^\infty\pr\left((Z_{h-1}^0)^{-1} \sum_{\|\tau\| = h} (W_+^\tau)^2 > \mu + \epsilon\right) 
%\nonumber  &\le& c \sum_{h=1}^\infty Var\left( (Z_{h-1}^0)^{-1} \sum_{\|\tau\| = h} (W_+^\tau)^2\right) \\
%&\le& c \sum_{h=1}^\infty \E\left( (Z_{h-1}^0)^{-2}\sum_{\|\tau\| = h} \E \left(((W_+^\tau)^2 - \mu )^2|\F_{h-1}\right)\right). \label{eq:lastone}
%\end{eqnarray}
Then,
\[\E \left( \left(W_+^2 -\mu\right)^2 | \F_{h-1} \right) \le \E(W_+^4) <c \epsilon^2 <\infty.\]
Using this to bound Equation \eqref{eq:lastone},
\[\le c \sum_{h=1}^\infty \E\left( \textbf{1}\{Z_{h-1} \ne 0\} Z_{h-1}^{-1}\right).\]
By Theorem 1 in  \cite{ney2003harmonic}, there exists some constant $\rho<1$ and some other constant $c$ such that 
\[E\left(\textbf{1}\{Z_{h} \ne 0\} Z_{h}^{-1}\right) \le c \rho^h.\]
Because this is a summable sequence, Borel-Cantelli implies the desired result.

\end{proof}

The following is a proof of Lemma \ref{lem:limW}.
\begin{proof}
Define 
\[W_{+,n} = \max_{0\le m \le n} W_m.\]
By the Monotone Convergence Theorem, $\E W_{+,n}^4 \rightarrow \E W_{+}^4$.  So, it is enough to show that \\ $\sup_n \E W_{+,n}^4 <\infty$.

By the $L^p$ maximum inequality (e.g. Theorem 5.4.3 in \cite{durrett}), $W_n = \E(W|\F_{n})$, and Jensen's inequality,
\[\E W_{+,n}^4 \le c \E W_{n}^4 = c \E (\E(W|\F_{n})^4) \le c \E(W^4).\]
%By the triangle inequality,
%\[\E W_{n}^4 \le \E |W_{n} - W|^4 + \E W^4.\]
%XXXXX Proposition 1.3 in \cite{liu2001local} shows that $m>1$ and $\E \xi^4<\infty$ imply that $\sup_n \E W_{n}^4 < \infty$, yielding the desired result.
%To see this, use 
%Moreover, he result in 
\cite{bingham1974asymptotic}  shows that $\E(\xi^4)<\infty$ implies $\E (W^4) <\infty$, concluding the proof.
%  By Jensen's inequality, $\E W_n^4$ is an increasing sequence.  Using the $L^p$ maximum inequality for Martingales (e.g. Theorem 5.4.3 in \cite{durrett}) and the monotone convergence theorem, 
%\[\E (W_+^\tau)^4 \le c \E (W^\tau)^4 <\infty.\]
%
%
\end{proof}

Next a proof of Theorem \ref{thm:critical}.

\begin{proof}  
First, a proof of the upper bound.  From Corollary \ref{cor:DEbounds}, $DE \le n \G_h(\lambda_2)$.  By the Kesten-Stigum Theorem, $\T$ grows at rate $m$. From Lemma  \ref{thm:gwp}, $\T$ is balanced.  So, Theorem \ref{thm:upperBound} gives upper bounds for $\G_h$.  Multiplying the bounds by $n$ yields the upper bound on $DE$ given in Equation \eqref{eq:critical}.  

For the lower bound, $DE \ge c n \G_h(\lambda_2)$ for a generic positive constant, $c > 0$.  By Fact \ref{fact:Gbounds}, $\G_h(\lambda_2) \ge \lambda_2^{2h}$.  By the Kesten-Stigum Theorem, $\T$ grows at rate $m$.  So, $n \ge \cb \sum_{k = 0}^h m^k \ge \cb m^h$.  So, $h\le \log_m n - c$.  Performing the algebra analogous to Equation \eqref{eq:Galpha}, yields $\G_h(\lambda_2) \ge c n^{- \log_m \lambda_2^{-2}}$. Multiplying by $n$ and combining this with Fact \ref{fact:oneovern} yields the lower bound.
\end{proof}

\section{Sampling with-replacement results in Section 4} \label{app:replacement}

The following is a proof of Proposition \ref{prop:lowerRn}.

\begin{proof}
% with stationary distribution $\pi_i \propto deg(i)$.  Define $\bar d =\frac{1}{n} \sum_{i \in G} deg(i)$.
%\[\pi_i = \frac{deg(i) }{\sum_j deg(j)} = \frac{deg(i)}{N \bar d} \le \frac{D}{N \bar d}.\]
%Rearranging terms,
%\[\frac{1}{deg(i)} \ge \frac{c}{\bar d}.\]
Let $\sigma'' \in \T$ denote a node that is distance two away from $\sigma \in \T$.  Let $\sigma' \in \T$ be the intermediate node between $\sigma$ and $\sigma''$.    Because $G$ is undirected and $P$ is a simple random walk,  $P$ is reversible.  So, the direction of the edges between $\sigma, \sigma',$ and $\sigma''$ does not matter.
% to denote the parent of node $\sigma \in \T$ and $\sigma''$ to denote the parent of $\sigma'$, 
\[\E(R_n) = \sum_{\sigma \ne \tau} \pr(X_\sigma = X_\tau) \ge \sum_{\sigma} \pr(X_\sigma = X_{\sigma''}) = \sum_\sigma \E\frac{1}{deg(X_{\sigma'})} \ge \sum_\sigma \frac{1}{D} = \frac{n}{D}.\]
%Define $c' = c/D$.
\end{proof}

The following is a proof of Theorem \ref{thm:upperRn}

\begin{proof}
Let $tr(P)$ denote the trace of $P$.
\[\pr(X_\sigma = X_\tau) 
= \sum_{i \in G} \pi_i \pr( X_\tau = i|X_\sigma = i) = \sum_{i \in G} \pi_i P_{ii}^{d(\sigma, \tau)} \le c N^{-1} tr(P^{d(\sigma, \tau)}) = c N^{-1} \sum_\ell \lambda_\ell^{d(\sigma, \tau)}\]

%
%Using properties of the trace function, 
%\begin{eqnarray}
%\pr(X_\sigma = X_\tau)
%&=& \sum_{i \in G} \pr(X_\sigma = i) \pr( X_\tau = i|X_\sigma = i) \\
%&=&  \sum_{i \in G} \pi_i P_{ii}^{d(\sigma, \tau)} \\
%&\le & cN^{-1}    tr(P^{d(\sigma, \tau)}) \\
%&=& c N^{-1} \sum_\ell \lambda_\ell^{d(\sigma, \tau)}.
%\end{eqnarray}
%the expected value of $R_n$ relates to quantities studied in the previous sections. 
%Then, suppose that $\pi_i \le c/N$ for all $i \in G$, 
Then,
\begin{eqnarray}
\E(R_n) &=& \sum_{\sigma \ne \tau}  \pr(X_\sigma = X_\tau)\\
&\le & c N^{-1} \sum_\ell \sum_{\sigma \ne \tau}   \lambda_\ell^{d(\sigma, \tau)}\\
&=& c N^{-1} \sum_\ell ( n^2 \G(\lambda_\ell) - n)\\
&=& c N^{-1} n^2\left( \G(\lambda_1) +  \sum_{\ell \in \mathscr{A}}   \G(\lambda_\ell) +   \sum_{\ell \in \mathscr{B}}    \G(\lambda_\ell) - c N/n \right),
\end{eqnarray}
where 
\[\mathscr{A} = \{ \ell > 1 :|\lambda_\ell| \ge m^{-1/2}\} \mbox{ and } \mathscr{B} = \{\lambda :|\lambda_\ell| < m^{-1/2}\}.\]
From properties of the Markov transition matrix, $\lambda_1 =1$.  So, $\G(\lambda_1) =1$.   By assumption (2), $|\mathscr{A}| \le k$ for some constant $k$. By Theorem \ref{thm:upperBound}, $\ell \in \mathscr{A}$ implies $\G(\lambda_\ell) = O( (\log n) n^{-\alpha}) $ for $\alpha = \log_m \lambda_2^{-2} > 0$.  Similarly, $\ell \in \mathscr{B}$ implies $\G(\lambda_\ell) = O( (\log n) n^{-1}) $.   Substituting these values, 
\[\E(R_n) \le \frac{cn^2}{N} +  k  (\log n)\frac{cn^{2-\alpha}}{N}  +   c (\log n)n . \]
By assumption (3), the first two terms converge to zero, leaving the third term which yields the result.
\end{proof}

\section{Simulation results; the widths of the confidence intervals} \label{sec:widths}

Figure \ref{fig:simWidth} examines the widths of the bootstrapped confidence intervals as a function of the sample size.   Using the class of simulation settings specified in Subsection \ref{sec:simsettings}, each point in  Figure \ref{fig:simWidth} represents an average over 501 simulations with $\lambda_2\approx.82$.   In all four panels, the horizontal axis is the sample size (with the spacing determined by the log scale). The top panels display the results for \texttt{aligned} $y$ and $z$.  The bottom panels give the results for \texttt{correlated} $y$ and $z$.  In the left two panels, the vertical axis is the width of the confidence interval (also on the log scale).  These panels include an additional line for the $5$th and $95$th percentile of the point estimate $\hat \mu_{VH}$ (over the 501 replicates); this line is labeled as the ``truth'' in the legend.  In the right two panels, the vertical axis represents the ratio of the width of the bootstrapped interval over the width of the truth.  

Figure \ref{fig:simWidth} shows that as the sample size increases, all of the widths decrease.  However, this simulation setting exceeds the critical threshold.  So, the width of the true interval does not decay at rate $O(1/\sqrt{n})$.  \atb and \utb detect this slower rate of convergence.  In the right panels, these lines are flat.  \achb does not detect this slower rate of convergence.  So, in the right panel, the line for \achb is decreasing.      In the top panels (\texttt{aligned} $y$), \ssb performs well.  However, in the bottom panel (\texttt{correlated} $y$), it contracts with \achbns.

\begin{figure}[h] %  figure placement: here, top, bottom, or page
   \centering
   \includegraphics[width=5in]{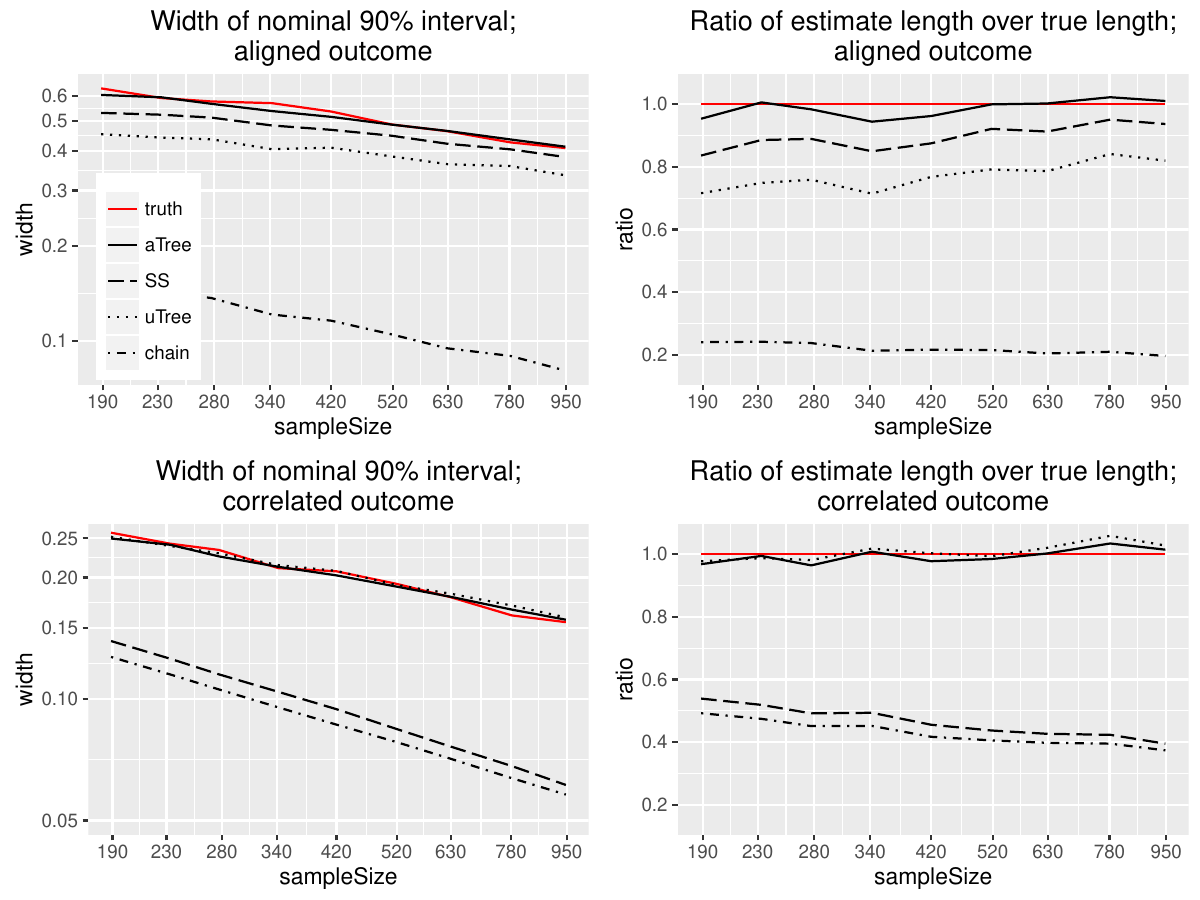} 
   \caption{In the top panels, the outcome $y$ is \texttt{aligned} with the referral bottleneck $z$.  In the bottom panels, the outcome $y$ is \texttt{correlated} with the referral bottleneck $z$.  \atb produces the widest confidence intervals.  \achb produces the narrowest confidence intervals.     In the left panels, the red is the distance between the $5$th and $95\%$ percentiles of the distribution of $\hat \mu_{VH}$ over the 501 replicates.  Because this is not a bootstrap interval, this is referred to as the truth.  In the right panels, the widths of the bootstrap intervals are divided by the width of the true line.  
   In the bottom right panel, \achb and \ssb have downward sloping lines, indicating that they are contracting more quickly than the truth.
   }
   \label{fig:simWidth}
\end{figure}

\begin{figure}[h] %  figure placement: here, top, bottom, or page
   \centering
   \includegraphics[width=4in]{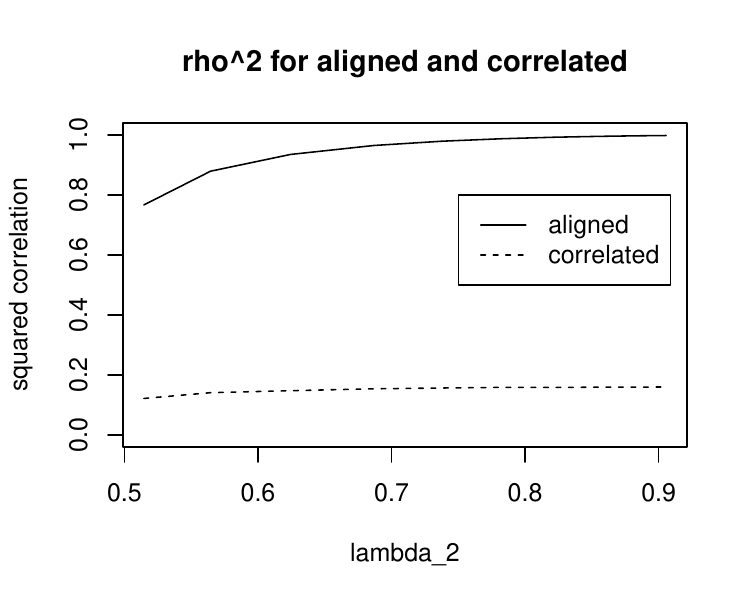} 
   \caption{  
  This figure plots  $\rho_\pi^2(y,f_2)$ as defined in
   Equation \eqref{def:rho}. This figure plots the value of $\rho_\pi^2(y,f_2)$ for both the \texttt{aligned} and \texttt{correlated} simulations.  The horizontal  axis corresponds to the different values of $\lambda_2$ examined in Figures  \ref{fig:2aligned}, \ref{fig:2correlated}, and \ref{fig:simWidth}. The \texttt{aligned} simulation has a value of $\rho_\pi^2(y,f_2)$ close to one, while the \texttt{correlated} simulation setting has a value of $\rho_\pi^2(y,f_2)$ around $.15$.  
%   Because $f_2$ is random, so is the exact value of $\rho_\pi^2(y,f_2)$.  
   }
   \label{fig:squaredCorrelation}
\end{figure}

\section{Bibliography}
\bibliographystyle{abbrvnat}
\bibliography{TV-references}

\end{document}